\documentclass{article}

\usepackage[utf8]{inputenc} 
\usepackage[T1]{fontenc}

\usepackage{amsmath}
\usepackage{amsfonts}
\usepackage{amssymb}
\usepackage{amsthm}
\usepackage{graphicx}
\graphicspath{ {images/} }
\usepackage{fancyhdr}
\usepackage{geometry}
\usepackage{hyperref}
\usepackage{mathrsfs}
\usepackage[nottoc,numbib]{tocbibind}
\usepackage{mathtools}
\usepackage{color}
\usepackage[normalem]{ulem}

\allowdisplaybreaks

\makeatletter
\DeclareRobustCommand{\format@sec@number}[2]{{\normalfont\upshape#1}#2}

\makeatother

\def\e{\varepsilon}

\def\a{\alpha}
\def\b{\beta}

\def\d{\delta}

\def\l{\lambda}

\def\R{\mathbb R}
\def\N{{\mathbb N}}

\def\Z{\mathbb Z}
\def\T{\mathbb T}

\def\C{\mathbb C}

\def\({\biggl(}
\def\){\biggr)}
\def\<{\mathbf{\langle}}
\def\>{\mathbf{\rangle}}

\numberwithin{equation}{section}

\newtheorem{theorem}[equation]{Theorem}
\newtheorem{proposition}[equation]{Proposition}
\newtheorem{lemma}[equation]{Lemma}

\newtheorem{que}{Question}

\theoremstyle{definition}
\newtheorem{definition}[equation]{Definition}

\theoremstyle{definition}

\theoremstyle{remark}
\newtheorem{remark}[equation]{Remark}

\title{Real-analytic diffeomorphisms with homogeneous spectrum and disjointness of convolutions}
\author{Shilpak Banerjee and Philipp Kunde}
\date{}

\graphicspath{{images/}} 

\begin{document}

\maketitle

\begin{abstract}
On any torus $\T^d$, $d \geq 2$, we prove the existence of a real-analytic diffeomorphism $T$ with a good approximation of type $\left(h,h+1\right)$, a maximal spectral type disjoint with its convolutions and a homogeneous spectrum of multiplicity two for the Cartesian square $T\times T$. The proof is based on a real-analytic version of the Approximation by Conjugation-method.
\end{abstract}


\section{Introduction}

One of the main problems in the spectral theory of dynamical systems at the interface of unitary operator theory and ergodic theory is the following question: 
\begin{que}
What are possible spectral properties for a Koopman operator associated with a measure-preserving transformation?
\end{que}
More specifically, one can search for transformations possessing specific essential values $\mathcal{M}_{U_T}$ of the spectral multiplicities: 
\begin{que} \label{que:multipl}
Given a subset $E \subset \mathbb{N} \cup \left\{ \infty \right\}$, is there an ergodic transformation $T$ such that $\mathcal{M}_{U_T}=E$?
\end{que}
These two problems are open and no restrictions (except for the obvious ones) are known. However, there is an impressive progress concerning Question \ref{que:multipl} (see \cite{Da} for a survey on spectral multiplicities of ergodic actions) and there exist two standard points of view: to consider the spectrum of $T$ (and in particular $\mathcal{M}_{U_T}$) either on $L^2\left(X, \mu\right)$ or on the orthogonal complement $L^2_0\left(X, \mu\right)$ of the constant functions. In \cite{KL} it was proved that all possible subsets of $\mathbb{N} \cup \left\{\infty\right\}$ can be realized as $\mathcal{M}_{U_T}$ for some ergodic transformation $T$ in the first case (since $1$ is always an eigenvalue because of the constant functions, ``possible'' means any subset of $\mathbb{N} \cup \left\{\infty\right\}$ with $1$ as an element). In the second case the Cartesian powers of a generic transformation provide a good opportunity for the construction of examples with the infimum of essential spectral multiplicities larger than $1$. Although, it seems very unlikely that these Cartesian powers have finite maximal spectral multiplicity, this is the generic case: Independently Ageev and Ryzhikov proved the celebrated result, that for a generic automorphism $T$ the Cartesian square $T \times T$ has homogeneous spectrum of multiplicity $2$ (see \cite{Ag} resp. \cite{Ry}). Ageev was even able to show that for the $n$-th power $T^n= T \times ... \times T$ of a generic transformation $T$ it holds $M\left(T^n\right)= \left\{n, n \cdot \left(n-1\right), ..., n!\right\}$ (cf. \cite[Theorem 2]{Ag}). He also proved for every $n \in \mathbb{N}$ the existence of an ergodic transformation with homogeneous spectrum of multiplicity $n$ in the orthogonal complement of the constant functions (\cite[Theorem 1]{Ag2}) solving Rokhlin's problem on homogeneous spectrum in ergodic theory. \\

Another important question in Ergodic Theory asks 
\begin{que} \label{que:smooth}
Are there smooth versions to the objects and concepts of abstract ergodic theory?
\end{que}
With regard to the \textit{spectral multiplicity problem} in Question \ref{que:multipl} there are only few results in this direction. Explicitly, Danilenko asks which subsets $E \neq \left\{ 1 \right\}$ admit a smooth ergodic transformation $T$ with $\mathcal{M}_{U_T}=E$ (\cite[section 10]{Da}). Blanchard and Lemanczyk showed that every set $E$ containing $1$ as well as $\text{lcm}(e_1,e_2)$ for $e_1,e_2 \in E$ is realizable as the set of essential spectral multiplicities for a Lebesgue measure-preserving analytic diffeomorphism of a finite dimensional torus (\cite{BL}). Recently, the second author proved that on any smooth compact connected manifold $M$ of dimension $d \geq 2$ admitting a non-trivial circle action $\mathcal{S} = \left\{\phi^t\right\}_{t \in \mathbb{S}^1}$ the set of $C^{\infty}$-diffeomorphisms $T$, that have a homogeneous spectrum of multiplicity $2$ for $T\times T$ and a maximal spectral type disjoint with its convolutions, is residual (i.e. it contains a dense $G_{\delta}$-set) in $\mathcal{A}_{\alpha}\left(M\right)= \overline{\left\{h \circ \phi^{\alpha} \circ h^{-1} \ : h \in \text{Diff}^{\infty}\left(M, \mu \right)\right\}}^{C^{\infty}}$ for every Liouvillean number $\alpha$ (\cite{Ku-Dc}). Hereby, Problem 7.11. in \cite{FK} was answered affirmatively. We are able to prove an analogous result in the real-analytic category,
\begin{theorem} \label{main}
For any $\rho>0$ and $d \geq 2$, there exist real-analytic diffeomorphisms $T\in \text{Diff }_\rho^\omega (\T^d, \mu)$ that have a maximal spectral type disjoint
with its convolutions, a homogeneous spectrum of multiplicity 2 for $T \times T$ and admit a good approximation of type $(h, h + 1)$.
\end{theorem}
The first part of this Theorem is linked to a conjecture of Kolmogorov respectively Rokhlin and Fomin (after verifying that the property held for all dynamical systems known at that time, especially large classes of systems of probabilistic origin like Gaussian ones), namely that every ergodic transformation possesses the so-called group property, i.e. the maximal spectral type $\sigma$ is symmetric and dominates its square $\sigma \ast \sigma$. This conjecture is an analogue of the well-known group property of the set of eigenvalues of an ergodic automorphism and was proven to be false. Indeed, in \cite{St1} A.M. Stepin gave the first example of a dynamical system without the group property. V.I. Oseledets constructed an analogous example with continuous spectrum (\cite{Os}). Later Stepin showed that for a generic transformation all convolutions $\sigma^k_0$, $k \in \mathbb{N}$, of the maximal spectral type $\sigma_0$ on $L^2_0\left(X, \mu\right)$ are mutually singular (see \cite{St}). \\

In general, one of the most powerful tools for finding answers to Question \ref{que:smooth} and for constructing volume preserving $C^{\infty}$-diffeomorphisms with prescribed ergodic or topological properties on any compact connected manifold $M$ of dimension $d\geq 2$ admitting a non-trivial circle action $\mathcal{S} = \left\{\phi^t\right\}_{t \in \mathbb{S}^1}$ is the so called \textit{Approximation by Conjugation}-method developed by D.V. Anosov and A. Katok in their fundamental paper \cite{AK}. These diffeomorphisms are constructed as limits of conjugates $T_n = H^{-1}_n \circ \phi^{\alpha_{n+1}} \circ H_n$, where $\alpha_{n+1} = \frac{p_{n+1}}{q_{n+1}}= \alpha_{n} + \frac{1}{k_{n} \cdot l_{n} \cdot q^2_{n}} \in \mathbb{Q}$, $H_n = h_n \circ H_{n-1}$ and $h_n$ is a measure-preserving diffeomorphism satisfying $\phi^{\alpha_{n}} \circ h_n = h_n \circ \phi^{\alpha_{n}}$. In each step the conjugation map $h_n$ and the parameter $l_{n}$ are chosen such that the diffeomorphism $T_n$ imitates the desired property with a certain precision. Then the parameter $k_n$ is chosen large enough to guarantee closeness of $T_{n}$ to $T_{n-1}$ in the $C^{\infty}$-topology and so the convergence of the sequence $\left(T_n\right)_{n \in \mathbb{N}}$ to a limit diffeomorphism is provided. This method enables the construction of smooth diffeomorphisms with specific ergodic properties (e.\,g. weak mixing ones in \cite[section 5]{AK}) or non-standard smooth realizations of measure-preserving systems (e.\,g. \cite[section 6]{AK} and \cite{FSW}). See also the very interesting survey article \cite{FK} for more details and other results of this method. 

Unfortunately, there are great challenging differences in the real-analytic category as discussed in \cite[section 7.1]{FK}: Since maps with very large derivatives in the real domain or its inverses are expected to have singularities in a small complex neighbourhood, for a real analytic family $S_t$, $0 \leq t \leq t_0$, $S_0 = \text{id}$, the family $h^{-1} \circ S_t \circ h$ is expected to have singularities very close to the real domain for any $t>0$. So, the domain of analycity for maps of our form $f_n = H^{-1}_n \circ \phi^{\alpha_{n+1}} \circ H_n$ will shrink at any step of the construction and the limit diffeomorphism will not be analytic. Thus, it is necessary to find conjugation maps of a special form which may be inverted more or less explicitly in such a way that one can guarantee analycity of the map and its inverse in a large complex domain.

Recently, some progress has been made in applying the AbC (short for approximation by conjugation)-method in the real-analytic category on particular manifolds. Fayad and Katok constructed volume-preserving uniquely ergodic real-analytic diffeomorphisms on odd-dimensional spheres in \cite{FK-ue}. In case of the torus $\mathbb{T}^d$, $d \geq 2$, the authors were able to reproduce several examples of smooth dynamical systems obtained by the AbC-scheme in the category of real-analytic diffeomorphisms in a series of papers (\cite{Ba-Ns}, \cite{Ku-Wm}, \cite{Ba-Ku}, \cite{Ba-Sr}). In particular, we construct minimal but not uniquely ergodic diffeomorphisms and nonstandard real-analytic realizations of toral translations. All these constructions base on the concept of \emph{block-slide type maps} on the torus and their sufficiently precise approximation by measure preserving real-analytic diffeomorphisms. This approach is the important mechanism in the constructions of this paper as well. We emphasize that all constructions in this article are done on the torus and that real-analytic AbC constructions on arbitrary real-analytic manifolds continue to remain an intractable problem.

\subsection*{Outline of the paper}
For a start we present several important definitions and concepts that will be used in this paper. In particular, we introduce the topology of real-analytic diffeomorphisms on the torus, the general AbC-scheme as well as block-slide type maps and their analytic approximations. In section \ref{section:theory} we give an overview of periodic approximation in ergodic theory. We have also included a survey section about spectral theory in Dynamical Systems. With the aid of this theoretical background we get a criterion for the proof of the main theorem in Proposition \ref{prop:red}. In fact, this Proposition reduces the proof to the construction of a diffeomorphism admitting a good cyclic approximation and a good linked approximation of type $(h,h+1)$. At this point, we will also sketch how to construct such a map. In particular, we describe the underlying combinatorics. In the rest of the paper, we present its construction via the concept of block-slide type maps and with explicitly defined tower elements in detail.

\section{Preliminaries}

Here we introduce the basic concepts and establish notations that we will use for the rest of this article. 

For a natural number $d$, we will denote the $d$ dimensional torus by $\T^d:=\R^d/\Z^d$. The standard Lebesgue measure on $\T^d$ will be denoted by  $\mu$. We define $\phi$, a measure preserving $\T^1$ action on the torus $\T^d$ as follows:
\begin{align}
\phi^t(x_1,\ldots, x_d)=(x_1+t,x_2,\ldots, x_d).
\end{align}

\subsection{The topology of real-analytic diffeomorphisms on the torus}

We give a description of the space of diffeomorphisms that are interesting to us.
Any real-analytic diffeomorphism on $\T^d$ homotopic to the identity admits a lift to a map from $\R^d$ to $\R^d$ and has the following form 
\begin{align} 
F(x_1,\ldots , x_d)=(x_1+f_1(x_1,\ldots, x_d),\ldots,x_d+f_d(x_1,\dots,x_d)),
\end{align}
where $f_i:\R^d\to \R$ are $\Z^d$-periodic real-analytic functions. Any real-analytic $\Z^d$-periodic function defined on $\R^d$ can be extended to some complex neighbourhood \footnote{we identify $\R^d$ inside $\C^d$ via the natural inclusion $(x_1,\ldots , x_d)\mapsto (x_1+i0,\ldots ,x_d+i0)$.} of $\R^d$  as a holomorphic (complex analytic) function. For a fixed $\rho>0$, let
\begin{align}
\Omega_\rho:=\{(z_1,\ldots,z_d)\in\C^d:|\text{Im}(z_1)|<\rho ,\ldots, |\text{Im}(z_d)|<\rho\}
\end{align}
 and for a function $f$ defined on this set, put 
 \begin{align}
 \|f\|_\rho:=\sup_{(z_1,\ldots, z_d)\in\Omega_\rho}|f(z_1,\ldots, z_d)|.
 \end{align}
 We define  $C^\omega_\rho(\T^d)$ to be the space of all $\Z^d$-periodic real-analytic functions on $\R^d$ that extends to a holomorphic function on $\Omega_\rho$ and $\|f\|_\rho<\infty$.

We define, $\text{Diff }^\omega_\rho(\T^d,\mu)$ to be the set of all measure-preserving real-analytic diffeomorphisms of $\T^d$ homotopic to the identity, whose lift $F(x)=(x_1+f_1(x),\ldots,x_d+f_d(x))$ to $\R^d$ satisfies $f_i\in C^\omega_\rho(\T^d)$ and we also require the lift $\tilde{F}(x)=(x_1+\tilde{f}_1(x),\ldots,x_d+\tilde{f}_d(x))$ of its inverse to $\R^d$ to satisfies $\tilde{f}_i\in C^\omega_\rho(\T^d)$.
The metric $d$ in $\text{Diff }^\omega_\rho(\T^d,\mu)$ is defined by 
\begin{align*}
d_\rho(f,g)=\max\{\tilde{d}_\rho(f,g),\tilde{d}_\rho(f^{-1},g^{-1})\}, \qquad\text{where}\qquad \tilde{d}_\rho(f,g)=\max_{i=1,\ldots, d}\{\inf_{n\in\Z}\|f_i-g_i+n\|_\rho\}.
\end{align*}
Let $F=(F_1,\ldots, F_d)$ be the lift of a diffeomorphism in $\text{Diff }^\omega_\rho(\T^d,\mu)$, we define the norm of the total derivative
\begin{align*}
\|DF\|_\rho:=\max_{\substack{i=1,\ldots, d\\j=1,\ldots, d}}\Big\|\frac{\partial F_i}{\partial x_j}\Big\|_\rho.
\end{align*}

Next, with some abuse of notation, we define the following two spaces 
\begin{align}
C^\omega_\infty (\T^d)  := & \cap_{n=1}^\infty C^\omega_n(\T^d), \label{6.789} \\
\text{Diff }^\omega_\infty (\T^d,\mu)  :=  &  \cap_{n=1}^\infty \text{Diff }^\omega_n(\T^d,\mu). \label{4.569}
\end{align}
Note that the functions in \ref{6.789} can be extended to $\C^d$ as entire functions. We also note that $\text{Diff }^\omega_\infty (\T^d,\mu)$ is closed under composition. To see this, let $f,g\in \text{Diff }^\omega_\infty (\T^d,\mu)$ and $F$ and $G$ be their corresponding lifts. Then note that $F\circ G$ is the lift of $f\circ g$ (with $\pi:\R^2\to\T^2$ as the natural projection, $\pi\circ F\circ G=f\circ\pi\circ G=f\circ g\circ \pi$). Now for the complexification of $F$ and $G$ note that the composition $F\circ G(z)=(z_1+g_1(z)+f_1(G(z)),\ldots, z_d+g_d(z)+f_d(G(z)) )$. Since $g_i\in C^\omega_\infty (\T^d) $, we have for any $\rho,$ $\sup_{z\in\Omega_\rho}|\text{Im}(G(z))|\leq \max_i (\sup_{z\in\Omega_\rho}|\text{Im}(z_i)+\text{Im}(g_i(z))|) \leq \max_i (\sup_{z\in\Omega_\rho}|\text{Im}(z_i)|+\sup_{z\in\Omega_\rho}|\text{Im}(g_i(z))|)\leq \rho + \max_i (\sup_{z\in\Omega_\rho}|g_i(z)|)<\rho + const<\rho'<\infty$ for some $\rho'$. So, $\sup_{z\in\Omega_\rho}|z_i+g_i(z)+f_i(G(z))|\leq |z_i|+|g_i(z)|+|f_i(G(z))|<\infty$ since $z\in\Omega_\rho, g_i\in C^\omega_\infty (\T^d)$ and $G(z)\in \Omega_{\rho'}$. An identical treatment gives the result for the inverse.

All intermediate diffeomorphisms constructed during the AbC method in this paper will belong to this category. \footnote{We note that the existence of such real-analytic functions whose complexification is entire or as in this case, the complexification of their lift is entire is central to a real-analytic AbC method. As of now we only know how to construct such functions on the torus, odd dimensional spheres and certain homogeneous spaces. }

This completes the description of the real-analytic topology necessary for our construction. Also throughout this paper, the word ``diffeomorphism'' will refer to a real-analytic diffeomorphism. Also, the word ``real-analytic topology'' will refer to the topology of $\text{Diff }^\omega_\rho(\T^d,\mu)$ described above. See \cite{S} for a more extensive treatment of these spaces.

\subsection{Some partitions of the torus}

First of all, we introduce the notion of a \emph{partial partition} of a measure space $\left(X,\mu\right)$, which is a pairwise disjoint countable collection of measurable subsets of $X$.
\begin{definition}
\begin{itemize}
	\item A sequence of partial partitions $\nu_n$ \emph{converges to the decomposition into points} if and only if for a given measurable set $A$ and for every $n \in \mathbb{N}$ there exists a measurable set $A_n$, which is a union of elements of $\nu_n$, such that $\lim_{n \rightarrow \infty} \mu \left( A \triangle A_n \right) = 0$. We often denote this by $\nu_n \rightarrow \varepsilon$.
	\item A partial partition $\nu$ is a \emph{refinement} of a partial partition $\eta$ if and only if for every $C \in \nu$ there exists a set $D \in \eta$ such that $C\subseteq D$. We write this as $\eta \leq \nu$.
\end{itemize}
\end{definition}
Using the notion of a partition we can introduce the weak topology in the space of measure-preserving transformations on a Lebesgue space:
\begin{definition}
\begin{enumerate}
	\item For two measure-preserving transformations $T,S$ and for a finite partition $\xi$ the \emph{weak distance} with respect to $\xi$ is defined by $d\left(\xi, T, S\right) \coloneqq \sum_{c \in \xi} \mu\left(T\left(c\right) \triangle S\left(c\right)\right)$.
	\item The base of neighbourhoods of $T$ in the \emph{weak topology} consists of the sets 
	\begin{equation*}
	W\left(T, \xi, \varepsilon\right) = \left\{S\;:\;d\left(\xi,T,S\right)< \varepsilon\right\},
	\end{equation*}
	where $\xi$ is a finite partition and $\varepsilon$ is a positive number.
\end{enumerate}
\end{definition}

There are some partitions of $\T^d$ that are of special interest to us. They appear repeatedly in this article and we summarize them here.

Assume that we are given three natural numbers $l,k,q$ and a function $a:\{0,1,\ldots,k-1\}\to\{0,1,\ldots,q-1\}$. We define the following three partitions of $\T^d$:
\begin{align}
& \mathcal{T}_{q}:=\Big\{\Delta_{i,q}:=\big[\frac{i}{q},\frac{i+1}{q}\big)\times \T^{d-1}: i = 0,1,\ldots,q-1\Big\},\label{partition T}\\
& \mathcal{G}_{l,q}:=\Big\{\big[\frac{i_1}{lq},\frac{i_1+1}{lq}\big)\times\big[\frac{i_2}{l},\frac{i_2+1}{l}\big)\times\ldots\times\big[\frac{i_d}{l},\frac{i_d+1}{l}\big):i_1 = 0,1,\ldots,lq-1,\nonumber\\
&\qquad\qquad\qquad\qquad\qquad\qquad\qquad\qquad\qquad\qquad (i_2,\ldots,i_{d}) \in \{0,1,\ldots,l-1\}^{d-1}\Big\},\label{partition G}\\
& \mathcal{G}_{j,l,q}:=\Big\{\big[\frac{i_1}{l^{d+1-j}q},\frac{i_1+1}{l^{d+1-j}q}\big)\times\big[\frac{i_2}{l},\frac{i_2+1}{l}\big)\times\ldots\times\big[\frac{i_{j}}{l},\frac{i_{j}+1}{l}\big)\times\T^{d-j}:i_1 = 0,1,\ldots,l^{d+1-j}q-1,\nonumber\\
&\qquad\qquad\qquad\qquad\qquad\qquad\qquad\qquad\qquad\qquad (i_2,\ldots,i_{d-j+1}) \in \{0,1,\ldots,l-1\}^{j-1}\Big\},\label{partition Gj}\\
&\mathcal{R}_{a,k,q}:=\Big\{R_{j,q}:=\phi^{j/q}\Big(\bigcup_{i=0}^{k-1}\Delta_{a(i)k+i,kq}\Big), j=0,\ldots,q-1\Big\}. \label{partition R}
\end{align} 
We note $\phi^\a$ acts on the partitions \ref{partition T}, \ref{partition G}, \ref{partition Gj} and \ref{partition R} as a permutation for any choice of $p$ when $\a=p/q$.

\subsection{Block-slide type maps and their real-analytic approximations}

We recall that a \emph{step function} on the unit interval is a finite linear combination of indicator functions on intervals. We define for $1\leq i,j\leq d$ and $i\neq j$, the following piecewise continuous map on the $d$ dimensional torus,
\begin{align}
\mathfrak{h}:\T^d\to\T^d\text{ defined by }\mathfrak{h}(x_1,\ldots,x_d):=(x_1,\ldots,x_{i-1}, x_i + s(x_j)\mod 1,x_{i+1},\ldots, x_d),
\end{align}
where $s$ is a step function on the unit interval. We refer to any finite composition of maps of the above kind as a \emph{block-slide type of map} on the torus.

Inspired by \cite{BK} the purpose of the section is to demonstrate that a block-slide type of map can be approximated extremely well by measure-preserving real-analytic diffeomorphisms outside a set of arbitrarily small measure. This can be achieved because step function can be approximated well by real-analytic functions whose complexification is entire.

\begin{lemma} \label{lemma approx}
Let $k$ and $N$ be two positive integer and $\b=(\b_0,\ldots,\b_{k-1})\in [0,1)^k$. Consider a step function of the form 
\begin{align*}
\tilde{s}_{\b,N}:[0,1)\to \R\quad\text{ defined by}\quad \tilde{s}_{\b,N}(x)=\sum_{i=0}^{kN-1}\tilde{\b}_i\chi_{[\frac{i}{kN},\frac{i+1}{kN})}(x).
\end{align*}
Here $\tilde{\b}_i:=\b_j$ where $j:=i\mod k$. Then, given any $\e>0$ and $\d>0$, there exists a periodic real-analytic function $s_{\b,N}:\R\to\R$ satisfying the following properties:
\begin{enumerate}
\item Entirety: The complexification of $s_{\b,N}$ extends holomorphically to $\C$.
\item Proximity criterion: $s_{\b,N}$ is $L^1$-close to $\tilde{s}_{\b,N}$. We can say more,
\begin{align}\label{nearness}
\sup_{x\in[0,1)\setminus F}|s_{\b,N}(x)-\tilde{s}_{\b,N}(x)|<\e.
\end{align}
\item Periodicity: $s_{\b,N}$ is $1/N$-periodic. More precisely, the complexification will satisfy,
\begin{align}\label{boundedness} 
s_{\b,N}(z+n/N)=s_{\b,N}(z)\qquad\forall\; z\in\C\text{ and }n\in\Z.
\end{align}
\end{enumerate}
Where $F=\cup_{i=0}^{kN-1}I_i\subset [0,1)$ is a union of intervals centred around $\frac{i}{kN},\;i=1,\ldots, kN-1$ and $I_0=[0,\frac{\d}{2kN}]\cup[1-\frac{\d}{2kN},1)$ and $\l(I_i)=\frac{\d}{kN}\;\forall\; i$. 
\end{lemma}

\begin{proof}
See \cite[Lemma 4.7]{Ba-Ns} and \cite[Lemma 3.6]{Ku-Wm}.
\end{proof}

Note that the condition \ref{boundedness} in particular implies 
\begin{align*}
\sup_{z: \text{Im}(z)<\rho}s_{\b,N}(z)<\infty\quad\forall\; \rho>0.
\end{align*}
Indeed, for any $\rho>0$, put $\Omega'_\rho=\{z=x+iy:x\in [0,1], |y|<\rho\}$ and note that entirety of $s_{\b,N}$ combined with compactness of $\overline{\Omega'_\rho}$ implies $\sup_{z\in\Omega'_\rho}|s_{\b,N}(z)|<C$ for some constant $C$. Periodicity of $s_{\a,N}$ in the real variable and the observation $\Omega_\rho=\cup_{n\in\Z}\left(\Omega'_\rho + n\right)$ implies that $\sup_{z\in\Omega_\rho}|s_{\b,N}(z)|<C$. We have essentially concluded that  $s_{\b,N}\in C^\omega_\infty(\T^1)$. 

Finally we piece together everything and demonstrate how a block-slide type of map on the torus can be approximated by a measure-preserving real-analytic diffeomorphism.

\begin{proposition} \label{proposition approximation}
Let $\mathfrak{h}:\T^d\to\T^d$ be a block-slide type of map which commutes with $\phi^{1/q}$ for some natural number $q$. Then for any $\e>0$ and $\d>0$, there exists a real-analytic diffeomorphism $h\in\text{Diff }^\omega_\infty(\T^d,\mu)$ satisfying the following conditions:
\begin{enumerate}
\item Proximity property: There exists a set $E\subset\T^d$ such that $\mu(E)<\d$ and $\sup_{x\in\T^d\setminus E}\|h(x)-\mathfrak{h}(x)\|<\e$. 
\item Commuting property: $h\circ\phi^{1/q}=\phi^{1/q}\circ h$.
\end{enumerate} 
In this case we say the the diffeomorphism $h$ is $(\e,\d)$-close to the block-slide type map $\mathfrak{h}$. 
\end{proposition}

\begin{proof}
See \cite[Proposition 2.22]{Ba-Ku}
\end{proof}

\subsection{Real-analytic AbC-method} \label{subsection abc method}

Our objective now is to recall the approximation by conjugation scheme developed by Anosov and Katok in \cite{AK}. Though we modify this scheme slightly to be more suitable for our purpose and fit the notations of our article we insist that the method presented here is almost identical to the original construction.

The AbC-method is an inductive process where a sequence of diffeomorphisms $T_n\in\text{Diff }^\omega_\infty(\T^d,\mu)$ is constructed inductively. The diffeomorphisms $T_n$ converge to some diffeomorphism $T$ $\in$ $\text{Diff }^\omega_\rho(\T^d,\mu)$. Additionally $T_n$ s are chosen carefully so that they satisfy some finite version of the desired property of $T$. 

We now give an explicit description. At the beginning of the construction we fix a constant $\rho>0$ and note that all parameters chosen will depend on this $\rho$. 

Assume that the construction has been carried out up to the $n-1$ th stage and we have the following information available to us:

\begin{enumerate}
\item We have sequences of natural numbers $\{p_m\}_{m=1}^n$, $\{q_m\}_{m=1}^n$, $\{k_m\}_{m=1}^{n-1}$, $\{l_m\}_{m=1}^{n-1}$, a sequence of functions $\{a_m:\{0,\ldots, l_m-1\}\to\{0,\ldots, q_m-1\}\}_{m=1}^{n-1}$ and sequences of numbers $\{\a_m\}_{m=1}^n$, $\{\e_m\}_{m=1}^{n-1}$. They satisfy the following conditions:
\begin{align}
p_{m}=k_{m-1}l_{m-1}q_{m-1}p_{m-1} + 1,\qquad q_{m}=k_{m-1}l_{m-1}q_{m-1}^2,\qquad \a_m = \frac{p_m}{q_m}, \qquad \e_m< 2^{-q_m}.
\end{align} 

\item The sequence of diffeomorphisms $\{T_m\}_{m=1}^{n-1}$ is constructed as conjugates of a periodic translation. More precisely,
\begin{align}
T_m:=H_m^{-1}\circ\phi^{\a_{m+1}}\circ H_m,\qquad\qquad H_m:=h_m\circ H_{m-1},\qquad\qquad h_m\in\text{Diff }^{\omega}_\infty(\T^d,\mu).
\end{align}
The diffeomorphisms $\{h_{m}\}_{m=1}^{n-1}$ satisfy the following commuting condition:
\begin{align}
h_{m}\circ\phi^{\a_{m}}=\phi^{\a_{m}}\circ h_{m}.
\end{align}

\item For $m =1,\ldots, n-1$, the diffeomorphism $T_m$ preserves and permutes two sequences of partitions, namely, $H_m^{-1}\mathcal{R}_{a_{m+1},l_{m+1},q_{m+1}}$ and $\mathcal{F}_{q_{m+1}}:=H_m^{-1}\mathcal{T}_{q_{m+1}}$.

\item For $m =1,\ldots, n-1$, $\mu(h_{m}^{-1}R_{i,q_{m}}\triangle\Delta_{i,q_{m}})<\e_{m}$ for any $R_{i,q_{m}}\in\mathcal{R}_{a_{m},l_{m},q_{m}}$ and $\Delta_{i,q_{m}}\in \mathcal{T}_{q_{m}}$ with the same $i$.

\item For $m =1,\ldots, n-1$, $\text{diam} (\mathcal{F}_{q_{m+1}}\cap E_{m})<\e_m$ \footnote{ This means that the diameter of the intersection of any atom of $\mathcal{F}_{q_m}$ and $E_m$ is less that $\e_m$.} for some measurable set $E_m$ satisfying $\mu(E_m)>1-\e_m$. (Note that this means $\mathcal{F}_{q_{m+1}}$  is a generating but not necessarily monotonic sequence of partitions.)

\item For $m =1,\ldots, n-1$: $d_\rho(T_m,T_{m-1})<\e_m$.
\end{enumerate}

Now we show how to do the construction at the $n$ th stage of this induction process. We proceed in the following order:
\begin{enumerate}
\item The number $l_n$ has to be a large enough integer so that the following conditions are satisfied:
\begin{equation}\label{eq:l}
l_n>d \cdot n^2 \cdot \|DH^{-1}_{n-1}\|_0,
\end{equation}
\begin{equation} \label{eq:leven}
2 \cdot l_{n-1} \cdot q_n \text{ divides } l_n.
\end{equation}

\item We choose $l_n $ and our function $a_n:\{0,\ldots, l_n-1\}\to\{0,\ldots, q_n-1\}$. This choice will depend on the construction we are doing and the specific properties we are targeting to prove.

\item Find a block-slide type map $\mathfrak{h}_{a_n,l_n,q_n}$ which commutes with $\phi^{\a_n}$, maps the partition $\mathcal{G}_{l_n,q_n}$ to $\mathcal{T}_{l_n^dq_n}$ and it maps the partition $\mathcal{T}_{q_n}$ to the partition $\mathcal{R}_{a_n,l_n,q_n}$. 

\item Use Proposition \ref{proposition approximation} to construct $h_{n}$ which is $(\e_n,\delta_n)$ close to $\mathfrak{h}_{a_n,l_n,q_n}$. Put $E_n$ to be the error set in Proposition \ref{proposition approximation}. In our constructions we will choose
\begin{equation}
    \delta_n = \frac{1}{n \cdot q_n}, \ \ \e_n = \frac{\delta_n}{4 \cdot l^d_n \cdot q^2_n}.
\end{equation}

\item By choosing $k_n$ sufficiently large we ensure that $|\a_{n+1}-\a_n|$ is small enough to guarantee $d_\rho(T_{n},T_{n-1})<\e_m$. 

\end{enumerate}
This completes the construction at the $n$ th stage. Note that this way convergence of $T_n$ to some $T\in\text{Diff }^\omega_\rho(\T^d,\mu)$ is guaranteed.  

\section{Periodic approximation in Ergodic Theory} \label{section:theory}
In this section we provide a short introduction to the concept of periodic approximation in Ergodic Theory. A more comprehensive presentation can be found in \cite{K}. \\
Let $\left(X, \mu\right)$ be a Lebesgue space. A \emph{tower} $t$ of height $h(t)=h$ is an ordered sequence of disjoint measurable sets $t=\left\{c_1,...,c_h\right\}$ of $X$ having equal measure, which is denoted by $m\left(t\right)$. The sets $c_i$ are called the \emph{levels} of the tower, especially $c_1$ is the \emph{base}. Associated with a tower there is a \emph{cyclic permutation} $\sigma$ sending $c_1$ to $c_2$, $c_2$ to $c_3$,... and $c_h$ to $c_1$. Using the notion of a tower we can give the next definition:
\begin{definition}
A \emph{periodic process} is a collection of disjoint towers covering the space $X$ and the associated cyclic permutations together with an equivalence relation among these towers identifying their bases.
\end{definition}
There are two partial partitions associated with a periodic process: The partition $\xi$ into all levels of all towers and the partition $\eta$ consisting of the union of bases of towers in each equivalence class and their images under the iterates of $\sigma$, where when we go beyond the height of a certain tower in the class we drop this tower and continue until the highest tower in the equivalence class has been exhausted. Obviously, we have $\eta \leq \xi$.\\
A sequence $\left(\xi_n, \eta_n, \sigma_n\right)$ of periodic processes is called \emph{exhaustive} if $\eta_n \rightarrow \varepsilon$. Such an exhaustive sequence of periodic processes is \emph{consistent} if for every measurable subset $A\subseteq X$ and each union of sets in $\eta_n$ approximating $A$ (i.\,e. $\mu\left(A_n \triangle A\right) \rightarrow 0$ as $n\rightarrow \infty$) the sequence $\sigma_n\left(A_n\right)$ converges to a set $B$, i.\,e. $\mu\left(\sigma_n\left(A_n\right) \triangle B\right) \rightarrow 0$ as $n\rightarrow \infty$. Moreover, we will call a sequence of towers $t^{(n)}$ from the periodic process $\left(\xi_n, \eta_n, \sigma_n\right)$ \emph{substantial} if there exists $r>0$ such that $h\left(t^{(n)}\right) \cdot m\left(t^{(n)}\right) >r$ for every $n \in \mathbb{N}$.
\begin{definition}
Let $T: \left(X, \mu\right)\rightarrow \left(X, \mu\right)$ be a measure-preserving transformation. An exhaustive sequence of periodic processes $\left(\xi_n,\eta_n, \sigma_n\right)$ forms a \emph{periodic approximation} of $T$ if 
\begin{equation*}
d\left(\xi_n,T,\sigma_n\right)=\sum_{c \in \xi_n} \mu\left(T\left(c\right) \triangle \sigma_n\left(c\right)\right) \rightarrow 0 \ \ \ \ \text{as } n\rightarrow \infty.
\end{equation*}
Given a sequence $g\left(n\right)$ of positive numbers we will say that the transformation $T$ admits a periodic approximation with \emph{speed} $g\left(n\right)$ if for a certain subsequence $\left(n_k\right)_{k \in \mathbb{N}}$ there exists an exhaustive sequence of periodic processes $\left(\xi_k,\eta_k,\sigma_k\right)$ such that $d\left(\xi_k,T,\sigma_k\right)<g\left(n_k\right)$.
\end{definition}
In order to define the type of the periodic approximation we need the notion of equivalence for sequences of periodic processes:
\begin{definition}
Two sequences of periodic processes $P_n = \left(\xi_n,\eta_n,\sigma_n\right)$ and $P'_n=\left(\xi'_n,\eta'_n,\sigma'_n\right)$ are called \emph{equivalent} if for every $n \in \mathbb{N}$ there is a bijective correspondence $\theta_n$ between subsets $S_n$ and $S'_n$ of the sets of towers of $P_n$ respectively $P'_n$ such that
\begin{itemize}
	\item For $t \in S_n$: $h\left(\theta_n\left(t\right)\right)=h\left(t\right)$.
	\item $\sum_{t \in S_n} h\left(t\right)m\left(t\right) \rightarrow 1$ as $n\rightarrow \infty$.
	\item $\sum_{t \in S_n} h\left(t\right) \cdot \left|m\left(t\right)-m\left(\theta_n\left(t\right)\right)\right|\rightarrow 0$ as $n\rightarrow \infty$.
	\item If two towers from $S_n$ are equivalent in $P_n$, then their images under $\theta_n$ are equivalent in $P'_n$.
\end{itemize}
\end{definition}
There are various types of approximation. We introduce the most important ones:
\begin{definition}
\begin{enumerate}
	\item A \emph{cyclic process} is a periodic process which consists of a single tower of height $h$. An approximation by an exhaustive sequence of cyclic processes is called a cyclic approximation. More specifically we will refer to a cyclic approximation with speed $o\left(\frac{1}{h}\right)$ as a \emph{good cyclic approximation}.
	\item An approximation generated by periodic processes equivalent to periodic processes consisting of two substantial towers whose heights differ by one is said to be of \emph{type $\left(h,h+1\right)$}. Equivalently the heights of the two towers $t_1$ and $t_2$ with base $B_1$ resp. $B_2$ are equal to $h$ and $h+1$ and for some $r>0$ we have $\mu\left(B_1\right) > \frac{r}{h}$ as well as $\mu\left(B_2\right)>\frac{r}{h+1}$. We will call the approximation of type $\left(h,h+1\right)$ with speed $o\left(\frac{1}{h}\right)$ \emph{good} and with speed $o\left(\frac{1}{h\cdot (h+1)}\right)$ \emph{excellent}.
	\item An approximation of type $\left(h,h+1\right)$ will be called a \emph{linked approximation of type $\left(h,h+1\right)$} if the two towers involved in the approximation are equivalent. This insures that the sequence of partitions $\eta_n$ generated by the union of the bases of the two towers and the iterates of this set converges to the decomposition into points.
\end{enumerate}
\end{definition}
\begin{remark} \label{rem:bem1}
As noted in \cite{Ry2} a good linked approximation of type $\left(h,h+1\right)$ implies the convergence
\begin{equation*}
U^{k \cdot \left(h+1\right)}_T \longrightarrow_{w} r \cdot U^k_T + \left(1-r\right) \cdot Id
\end{equation*}
in the weak operator topology for every $k \in \mathbb{N}$ and some $r \in \left(0,1\right)$, where $U_T$ is the Koopman-operator of $T$ (see section \ref{subsection:specm}).
\end{remark}
From the different types of approximations various ergodic properties can be derived. For example in \cite[Corollary 2.1.]{KS1} the subsequent Lemma is proven.
\begin{lemma} \label{lem:erg}
Let $T: \left(X, \mu\right)\rightarrow \left(X, \mu\right)$ be a measure-preserving transformation. If $T$ admits a good cyclic approximation, then $T$ is ergodic.
\end{lemma}
In \cite{KS} Katok and Stepin proved the genericity of automorphisms having a continuous spectrum in the set of measure-preserving homeomorphisms (recall that a transformation has a continuous spectrum, i.e. the corresponding operator $U_T$ in the space $L^2\left(M,\mu\right)$ has no eigenfunctions other than constants, if and only if it is weakly mixing). For this purpose, they deduced the following result:
\begin{lemma}[\cite{KS}, Theorem 5.1.] \label{lem:wm}
Let $T: \left(X, \mu\right)\rightarrow \left(X, \mu\right)$ be a measure preserving transformation. If $T$ is ergodic and admits a good approximation of type $\left(h,h+1\right)$, then $T$ has continuous spectrum.
\end{lemma}

\section{Spectral theory of dynamical systems}
Besides the concept of periodic approximation we will need further mathematical tools. We refer to \cite{Na} and \cite{Go} for more details.

\subsection{Spectral types} \label{subsection:specm}
Let $\left(X, \mu\right)$ be a Lebesgue space and $T:\left(X, \mu\right) \rightarrow \left(X, \mu\right)$ be an automorphism. Then we define the induced \emph{Koopman-operator} $U_T: L^2\left(X, \mu\right) \rightarrow L^2\left(X, \mu\right)$ by $U_T f = f \circ T$. Since
\begin{equation*}
\left\langle U_T f, U_T g \right\rangle = \int_{X} f \circ T \cdot \overline{g \circ T} \:d\mu = \int_{X} f  \cdot \overline{g } \:d\mu = \left\langle f,g\right\rangle \ \ \ \text{ for every } f,g \in L^2\left(X;\mu\right)
\end{equation*}
and $U^{-1}_T = U_{T^{-1}}$ this is an unitary operator on the Hilbert space $L^2\left(X,\mu\right)$.
\begin{remark} \label{rem:conj}
If two measure-preserving dynamical systems $\left(X_1, \mu_1, T_1\right)$ and $\left(X_2, \mu_2, T_2\right)$ are metrically isomorphic, their isomorphism $h: X_1 \rightarrow X_2$ induces an isomorphism of Hilbert spaces $V_h: L^2\left(X_2, \mu_2\right) \rightarrow L^2\left(X_1, \mu_1\right)$ by $\left(V_hf\right) = f \circ h$. Then we have $U_{T_1} = V_h \circ U_{T_2} \circ V^{-1}_h$ and this relation is called \emph{unitary equivalence} of operators. Hence, any invariant of unitary equivalence defines an invariant of isomorphisms. Such invariants are said to be \emph{spectral invariants} or spectral properties. \\
Moreover, we note that $1$ is always an eigenvalue of $U_T$ because of the constant functions. So when we discuss the spectral properties of $U_T$ we refer to its spectral properties that are restricted to the orthogonal complement of the constants. Hence, we consider the properties of $U_T$ in the space $L^2_0\left(X, \mu\right)$ of all $L^2$-functions with zero integral.
\end{remark}
One of the important spectral invariants are the so-called \emph{spectral measures}: Let $f \in L^2_0\left(X, \mu\right)$ and $Z\left(f\right) \coloneqq \overline{\text{span}\left\{U^n_T f \;: \; n \in \mathbb{Z} \right\}}^{L^2_0\left(X,\mu\right)}$. Using Bochner's theorem one can prove the existence of a finite Borel measure $\sigma_f$ defined on the unit circle $\mathbb{S}^1$ in the complex plane satisfying
\begin{equation*}
\left\langle U^n_T f,f\right\rangle = \int_{\mathbb{S}^1} z^n \:d\sigma_f(z) \ \ \ \ \text{ for every } n \in \mathbb{Z}.
\end{equation*}
Then $\sigma_f$ is called the spectral measure of $f$ with respect to $U_T$. \\
Moreover, by the Hahn-Hellinger Theorem, there is a sequence of functions $f_n \in L^2_0\left(X, \mu\right)$, $n \in \mathbb{N}$, for which
\begin{equation*}
L^2_0\left(X, \mu\right) = \oplus_{n \in \mathbb{N}} Z\left(f_n\right) \ \ \  \text{and} \ \ \ \sigma_{f_1} \gg \sigma_{f_2} \gg ...
\end{equation*}
These measures are unique in the sense that for any other family of functions $g_n \in L^2_0\left(X, \mu\right)$, $n \in \mathbb{N}$, for which $L^2_0\left(X, \mu\right) = \oplus_{n \in \mathbb{N}} Z\left(g_n\right)$ and $\sigma_{g_1} \gg \sigma_{g_2} \gg ...$ we have $\sigma_{f_n} \sim \sigma_{g_n}$ for every $n \in \mathbb{N}$.
\begin{definition}
The spectral type of $\sigma_{f_1}$ is called the \emph{maximal spectral type} $\sigma$ of $U_T$.
\end{definition}
According to this we say that $U_T$ has a continuous spectrum if $\sigma_{f_1}$ is a continuous measure and $U_T$ has a discrete spectrum if $\sigma_{f_1}$ is a discrete measure.

\subsection{Spectral multiplicities}
Besides the maximal spectral type an important characterization of $U_T$ is the \emph{multiplicity function} $M_{U_T}: \mathbb{S}^1 \rightarrow \mathbb{N} \cup \left\{\infty\right\}$, which is $\sigma_{f_1}$- almost everywhere defined by 
\begin{equation*}
M_{U_T}\left(z\right) = \sum^{\infty}_{i=1} \chi_{A_i}\left(z\right), \text{ where } A_i = \left\{ z \in \mathbb{S}^1 \; : \; \frac{d\sigma_{f_i}}{d\sigma_{f_1}}\left(z\right) > 0\right\}.
\end{equation*}
Here $\frac{d\sigma_{f_i}}{d\sigma_{f_1}}$ is the Radon-Nikodym derivative of $\sigma_{f_i}$ with respect to $\sigma_{f_1}$. \\
Using this multiplicity function we establish the set $\mathcal{M}_{U_T}$ of \emph{essential spectral multiplicities}, which is the essential range of $M_{U_T}$ with respect to $\sigma_{f_1}$. Then we define the \emph{maximal spectral multiplicity} $m_{U_T}$ as the essential supremum (with respect to $\sigma_{f_1}$) of $\mathcal{M}_{U_T}$.
\begin{definition}
$U_T$ is said to have \emph{homogeneous spectrum} of multiplicity $m$ if $\mathcal{M}_{U_T}= \left\{m\right\}$. In particular, $U_T$ has a \emph{simple spectrum} if $\mathcal{M}_{U_T}= \left\{1\right\}$. In all other cases $U_T$ has a non-simple spectrum.
\end{definition}

In connection with the previous chapter \ref{section:theory} we state the following result (\cite[Theorem 3.1.]{KS1}):
\begin{lemma} \label{lem:simp}
Let $T$ be an automorphism of a Lebesgue space. If $T$ admits a cyclic approximation of speed $\frac{\theta}{h}$, where $\theta < \frac{1}{2}$, then the spectrum of $U_T$ is simple.
\end{lemma}
For automorphisms with simple spectrum we have the subsequent theorem of Ryzhikov (\cite[Theorem 2.1.]{Ry1}):
\begin{lemma} \label{lem:ryz}
Let $\left(X, \mu\right)$ be a Lebesgue probability space and $T: \left(X, \mu\right) \rightarrow \left(X, \mu\right)$ be an automorphism with simple spectrum. Suppose that the weak convergence
\begin{equation*}
U^{k_n}_T \longrightarrow_{w} \left( a \cdot U_T + \left(1-a\right) \cdot Id\right)
\end{equation*}
holds for some $a \in \left(0,1\right)$ and some strictly increasing sequence $\left(k_n\right)_{n \in \mathbb{N}}$ of natural numbers. \\
Then the Cartesian square $T \times T$ has a homogeneous spectrum of multiplicity $2$.
\end{lemma}

\subsection{Disjointness of convolutions}
In this section we study the convolutions of the maximal spectral type $\sigma$. Therefore, we state the definition of a \emph{convolution of measures}:
\begin{definition}
Let $G$ be a topological group and $\mu$, $\nu$ finite Borel measures on $G$. Then their convolution $\mu \ast \nu$ is defined by
\begin{equation*}
\left(\mu \ast \nu\right)\left(A\right) = \int\int 1_A\left(x\cdot y\right) \:d\mu(x)\: d\nu(y)
\end{equation*}
for each measurable set $A$ of $G$.
\end{definition}
If all the convolutions $\sigma^k = \sigma \ast ... \ast \sigma$ for $k \in \mathbb{N}$ are pairwise mutually singular, one speaks about \emph{disjointness of convolutions}. To guarantee this pairwise singularity of convolutions of the maximal spectral type of a measure-preserving transformation the following property is useful:
\begin{definition}
An automorphism $T$ of a Lebesgue space $\left(X,\mu\right)$ is said to be \emph{$\kappa$-weakly mixing}, $\kappa \in \left[0,1\right]$, if there exists a strictly increasing sequence $\left(k_n\right)_{n \in \mathbb{N}}$ of natural numbers such that the weak convergence
\begin{equation*}
U^{k_n}_T \longrightarrow_{w} \left(\kappa \cdot P_c + \left(1- \kappa\right) \cdot Id\right)
\end{equation*}
holds, where $P_c$ is the projection on the subspace of constants.
\end{definition}
\begin{remark} \label{rem:geometric}
By \cite[Proposition 3.1.]{St} we can characterise this property in geometric language: A transformation $T$ is $\kappa$-weakly mixing if and only if there is an increasing sequence $\left(k_n\right)_{n \in \mathbb{N}}$ of natural numbers such that for all measurable sets $A$ and $B$ 
\begin{equation*}
\lim_{n \rightarrow \infty} \mu\left(A \cap T^{k_n}B\right) = \kappa \cdot \mu\left(A\right) \cdot \mu\left(B\right) + \left(1-\kappa\right) \cdot \mu\left(A \cap B\right).
\end{equation*}
We recognize that $0$-weak mixing corresponds to rigidity and $1$-weak mixing to the usual notion of weak mixing.
\end{remark}
As announced this property has connections with certain properties of the maximal spectral type (see \cite[Theorem 1]{St}):
\begin{lemma} \label{lem:kappa}
If the transformation $T$ is $\kappa$-weakly mixing for some $0<\kappa <1$ and $\sigma$ is the maximal spectral type for $U_T|_{L^2_0\left(X,\mu\right)}$, then $\sigma$ and all its convolutions $\sigma^k = \sigma \ast ... \ast \sigma$ are pairwise mutually singular.
\end{lemma}

\section{Criterion for the proof of Theorem \ref{main}}

\subsection{Statement of the criterion}
As discussed in \cite[section 9]{Ku-Dc} we can reduce the proof of the Theorem to Proposition \ref{prop:red}. 
\begin{proposition} \label{prop:red}
Let $f \in \text{Diff}^{\omega}_{\rho}\left( \T^d, \mu \right)$ admit a good linked approximation of type $\left(h,h+1\right)$ and a good cyclic approximation. Then $f$ has a maximal spectral type disjoint with its convolutions and homogeneous spectrum of multiplicity $2$ for its ergodic Cartesian square $f \times f$
\end{proposition}
\begin{proof}
Indeed, such a constructed diffeomorphism has the following properties:
\begin{itemize}
	\item Since $f$ allows a good cylic approximation, it has simple spectrum by Lemma \ref{lem:simp}. Furthermore, $f$ admits a good linked approximation of type $\left(h,h+1\right)$ and so we can use Remark \ref{rem:bem1} to obtain the weak convergence $U^{h+1}_f \longrightarrow_{w} r \cdot U_f + \left(1-r\right) \cdot Id$ for some $r \in \left(0,1\right)$. Hence, we can apply Lemma \ref{lem:ryz} and conclude that $f \times f$ has homogeneous spectrum of multiplicity $2$.
	\item With the aid of Lemma \ref{lem:erg} and the good cyclic approximation of $f$ we have that $f$ is ergodic. Then we can exploit the good approximation of type $\left(h,h+1\right)$ and Lemma \ref{lem:wm} to see that $f$ is even weakly mixing (hence, $f \times f$ is ergodic). Due to Remark \ref{rem:geometric} there exists a strictly increasing sequence $\left(k_n\right)_{n \in \mathbb{N}}$ of natural numbers such that the convergence $U^{k_n}_f \rightarrow_{w} P_c$ holds in the weak operator topology as $n\rightarrow \infty$. Along this sequence we have using Remark \ref{rem:bem1}
	\begin{equation*}
	U^{k_n \cdot \left(h+1\right)}_f \longrightarrow_w \left(r \cdot P_c + \left(1-r\right) \cdot Id\right)
	\end{equation*}
	for some $r \in \left(0,1\right)$. Thus, $f$ is $\kappa$-weakly mixing (with $\kappa = r \in \left(0,1\right)$). Then the maximal spectral type $\sigma$ of $f$ is disjoint with its convolutions by Lemma \ref{lem:kappa}.
\end{itemize}
\end{proof}

\subsection{Sketch of the construction}
As announced we would like to sketch the construction of $T_n = H^{-1}_n \circ \phi^{\a_{n+1}} \circ H_n$ with $\a_{n+1} = \a_n + \frac{1}{k_n \cdot l_n \cdot q^2_n}$ and its combinatorics satisfying the requirements of Proposition \ref{prop:red}. \\
First of all, we define $m_n \coloneqq  \frac{q_{n+1}}{2 \cdot q^2_n}$. Additionally, we introduce two sets $\tilde{c}^{(n)}_{0,j}$, $j=1,2$, in section \ref{subsection:tower1}. With the aid of these we define the bases $c^{(n)}_{0,j}= H^{-1}_{n} \left(\tilde{c}^{(n)}_{0,j} \right)$ as well as levels $T^i_n\left( c^{(n)}_{0,j}\right) = H^{-1}_n \circ \phi^{i \cdot \a_{n+1}} \left(\tilde{c}^{(n)}_{0,j} \right)$ of the two towers of heights $m_n$ and $m_n+1$ respectively. These towers will be used to prove that $T= \lim_{n \rightarrow \infty} T_n$ admits a good linked approximation of type $(h,h+1)$. In particular, $\tilde{c}^{(n)}_{0,j}$ is built as a union of sets due to different reasons which we will explain in the following. First of all, the union over $i_2, \dots, i_d$ is taken such that $\tilde{c}^{(n)}_{0,j}$ has almost full length in the $x_2, \dots , x_d$-coordinates and at the same time is positioned in domains, where we can approximate our block slide type of map with a real-analytic volume-preserving diffeomorphism very well. This will be very useful in the calculations of the speed of approximation, especially in the proof of Lemma \ref{lem:error2}. In order to see that $T$ admits an approximation of type $(h,h+1)$ the so-called ``$i_1$-stripes'', i.\,e. sets obtained by the union over $i_2,\dots,i_d$, are very important. These stripes are depicted in the schematic visualisation of the tower bases in Figure \ref{fig:t}. In this connection we note that the number $m_n$ is defined such that $m_n \cdot \alpha_{n+1} = \frac{r_n}{q_n}+\frac{1}{2q^2_n}$ with some $r_n \in \N$. Hence, under $\phi^{m_n \cdot \alpha_{n+1}}$ a ``$i_1$-stripe'' is mapped into another $\frac{1}{q_n}$-sector and is shifted by $\frac{1}{2q^2_n}$. In order to get recurrence in the base we need another stripe to be positioned there, which will be fulfilled by our choice of the term $\frac{i_1 \cdot r_n}{q_n}$ and the union over $i_1$ in the definition of $\tilde{c}^{(n)}_{0,1}$. Similarly, we proceed to define $\tilde{c}^{(n)}_{0,2}$. So the sets $\tilde{c}^{(n)}_{0,j}$ will be defined in such a way that a great portion of $\phi^{m_n \cdot \alpha_{n+1}}\left(\tilde{c}^{(n)}_{0,1}\right)$ (resp. $\phi^{\left(m_n+1\right) \cdot \alpha_{n+1}} \left(\tilde{c}^{(n)}_{0,2}\right)$) is mapped back into $\tilde{c}^{(n)}_{0,1}$ (resp. $\tilde{c}^{(n)}_{0,2}$) for the prescribed height $m_n$ (resp. $m_n+1$) of the particular tower. Since both towers have to be substantial, the measure of each base element has to be about $\frac{1}{2m_n}$. Additionally, the tower levels have to be disjoint sets. We ensure this by taking a $x_1$-width of $\frac{q_n}{q_{n+1}}$ (note that under $\phi^{\alpha_{n+1}}$ we return to the same $\frac{1}{q_n}$-section after $q_n$ steps and then we have a $x_1$-deviation of $q_n \cdot \left(\a_{n+1}-\a_n \right)= \frac{q_n}{q_{n+1}}$). Moreover, the conjugation map $h_n$ will be defined such that each image $h^{-1}_n\left(\tilde{c}^{(n)}_{0,j}\right)$ is contained in a cube of edge length $\frac{1}{l_n}$. Hereby, the diameter of $c^{(n)}_{0,j}$ is small due to condition \ref{eq:l}. In Figure \ref{fig:h} we sketch this action of the conjugation map $h_n$. \\
The proof that $T$ admits a good cyclic approximation is considerably easier. Since $T^{q_{n+1}}_n = \text{id}$ we can use a subset with measure about $\frac{1}{q_{n+1}}$ of our tower base in the $(h,h+1)$-approximation.

\section{Explicit constructions}
In this section, we present the construction of the conjugation map $h_n$. We start by describing the combinatorial picture in subsection \ref{subsec:descrip}. In order to find a real-analytic conjugation map in subsection \ref{subsec:constr}, we recall the result from \cite[section 5.1]{Ba-Ku} that a permutation $\Pi$ of the partition 
\begin{align}
\mathcal{S}_{kq,l}:=\{S_{i,j}^{kq,l}:=[i/(kq),(i+1)/(kq))\times[j/l,(j+1)/l), 0\leq i< kq, 0\leq j< l\}
\end{align}
which commutes with $\phi^{1/q}$ is a block slide type of map.

\subsection{Description of the required combinatorics} \label{subsec:descrip}

\begin{figure}
    \centering
    \includegraphics[width=15cm, height=13.5cm, trim={1.5cm 6.5cm 6cm 2.15cm}, clip]{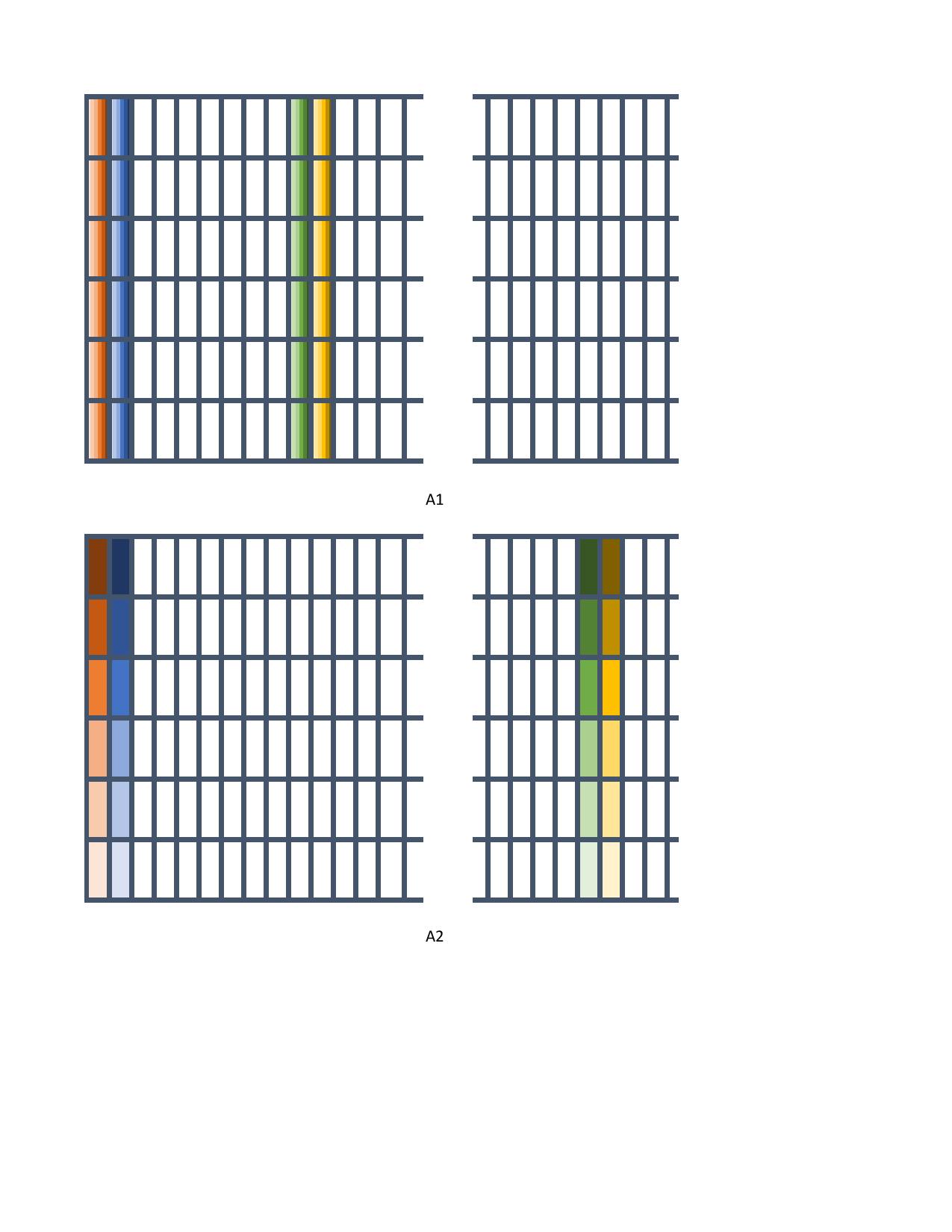}
    \caption{Illustration of the action of $\mathfrak{h}_{l,p,q,r}$ taking A1 to A2. It is drawn with $p=1, q=3, l=6, k=2$. Here $\T^2$ is partitioned into rectangles of horizontal length $1/(2l^2q^2)$ and vertical length $1/l$ and is colored with a shade of a certain color. For example the one on the left of A1 is colored with a light shade of red, and it progressively gets darker as we move to the right. Note that in A1, all the rectangles are colored by the same shade and color as we move vertically. As we move $6$ blocks horizontally to the right, one for each $f, 0\leq f<6$, the color changes from red to a shade of blue and the same pattern repeats. Under the action of $\mathfrak{h}_{l,p,q,r}$, A1 gets mapped to A2 and, where the rectangles gets rearranged so the horizontal progression to the right becomes a vertical progression to top, and hence the shade changes likewise. Pattern repeats for blue.
    Note that the green and and the yellow rectangles corresponds to $e\geq q$ and hence the behaviour changes, in fact they do not get mapped to the rectangle corresponding to the same $a$, rather to the rectangle corresponding to $a+e(r+p)$. However the pattern remains the same. Because of the larger size, we show this phenomenon with a break in the torus.}
    \label{fig:h}
\end{figure}

We describe the underlying combinatorics of our construction. For this purpose, we recall our choice of the rotation number
\begin{equation} \label{alpha}
  \a_{n+1} = \a_n + \frac{1}{l_n \cdot k_n \cdot q^2_n}. 
\end{equation}
Additionally, we define the number
\begin{equation}
    m_n = \frac{q_{n+1}}{2q^2_n}=\frac{l_n \cdot k_n}{2}.
\end{equation}
By equations \ref{alpha} and \ref{eq:leven} we have $m_n \in \N$. In fact, $m_n$ will be the height of the first tower and $m_n +1$ the height of the second one. We put
\begin{equation*}
    r_n \coloneqq m_n \cdot p_n \mod q_n.
\end{equation*}

Altogether we calculate
\begin{align}
m_n \cdot \a_{n+1} & = \frac{r_n}{q_n} + \frac{1}{2q^2_n}, \label{a1} \\
\left( m_n +1 \right) \cdot \a_{n+1} & = \frac{r_n + p_n}{q_n} + \frac{1}{2q^2_n}+\frac{1}{q_{n+1}}. \label{a2}
\end{align}

In the following, we consider the integers $0 \leq a < q$, $0\leq b < 2q$, $0 \leq c < \frac{l}{2q}$ (we recall from equation \ref{eq:leven} that $ \frac{l}{2q} \in \N$ in case of our constructions), $0 \leq e < 2q$, $0 \leq f < l$ and $0 \leq t_i < l$ for $i=1,\dots, d-2$. Under the map $\mathfrak{h}_{l,p,q,r}$ (see figure \ref{fig:h}) the set
\begin{equation*}
\begin{split}
  A_{a,b,c,t_1,\dots,t_{d-2}, e,f,j} \coloneqq  & \Bigg[ \frac{a}{q} + \frac{b}{2q^2}+\frac{c}{lq} + \frac {t_1}{l^2q} + \dots + \frac{t_{d-2}}{l^{d-1}q}+\frac{e}{2l^{d-1}q^2}+\frac{f}{2l^dq^2}, \\
  & \quad \frac{a}{q} + \frac{b}{2q^2}+\frac{c}{lq} + \frac {t_1}{l^2q} + \dots + \frac{t_{d-2}}{l^{d-1}q}+\frac{e}{2l^{d-1}q^2}+\frac{f+1}{2l^dq^2} \Bigg] \\
     & \times\left[ \frac{j}{l}, \frac{j+1}{l} \right] 
\end{split}\end{equation*}
is mapped to
\begin{equation*}
\begin{split}
    & \Bigg[ \frac{a+e \cdot r}{q} + \frac{e}{2q^2}+\frac{c}{lq} + \frac {t_1}{l^2q} + \dots + \frac{t_{d-2}}{l^{d-1}q}+\frac{b}{2l^{d-1}q^2}+\frac{j}{2l^dq^2}, \\
    & \quad  \frac{a+e \cdot r}{q} + \frac{e}{2q^2}+\frac{c}{lq} + \frac {t_1}{l^2q} + \dots + \frac{t_{d-2}}{l^{d-1}q}+\frac{b}{2l^{d-1}q^2}+\frac{j+1}{2l^dq^2} \Bigg] \times \left[ \frac{f}{l}, \frac{f+1}{l} \right] 
\end{split} 
\end{equation*}
in case of $0\leq e < q$ and to
\begin{equation*}
\begin{split}
    & \Bigg[ \frac{a+e \cdot (r+p)}{q} + \frac{e}{2q^2}+\frac{c}{lq} + \frac {t_1}{l^2q} + \dots + \frac{t_{d-2}}{l^{d-1}q}+\frac{b}{2l^{d-1}q^2}+\frac{j}{2l^dq^2}, \\
    & \quad  \frac{a+e \cdot (r+p)}{q} + \frac{e}{2q^2}+\frac{c}{lq} + \frac {t_1}{l^2q} + \dots + \frac{t_{d-2}}{l^{d-1}q}+\frac{b}{2l^{d-1}q^2}+\frac{j+1}{2l^dq^2} \Bigg] 
    \times \left[ \frac{f}{l}, \frac{f+1}{l} \right] 
\end{split} 
\end{equation*}
in case of $q \leq e < 2q$.

\begin{remark} \label{remark combi}
We note several important properties of the map $\mathfrak{h}_{l,p,q,r}$:
\begin{enumerate}
\item $\mathfrak{h}_{l,p,q,r} \circ \phi^{\frac{1}{q}} = \phi^{\frac{1}{q}} \circ \mathfrak{h}_{l,p,q,r}$.
\item In case of $0 \leq e < q$ we have
\begin{equation*} \begin{split}
    & \mathfrak{h}_{l,p,q,r} \left( \bigcup^{l-1}_{f=0} A_{a,b,c,t_1,\dots,t_{d-2},e,f,j} \right) \\
    = & \Bigg[ \frac{a+e \cdot r}{q} + \frac{e}{2q^2}+\frac{c}{lq} + \frac {t_1}{l^2q} + \dots + \frac{t_{d-2}}{l^{d-1}q}+\frac{b}{2l^{d-1}q^2}+\frac{j}{2l^dq^2}, \\
    & \quad \frac{a+e \cdot r}{q} + \frac{e}{2q^2}+\frac{c}{lq} + \frac {t_1}{l^2q} + \dots + \frac{t_{d-2}}{l^{d-1}q} +\frac{b}{2l^{d-1}q^2}+\frac{j+1}{2l^dq^2} \Bigg] \times \left[ 0,1 \right] 
\end{split}\end{equation*}
and in case of $q \leq e < 2q$ we have 
\begin{equation*} \begin{split}
    & \mathfrak{h}_{l,p,q,r} \left( \bigcup^{l-1}_{f=0} A_{a,b,c,t_1,\dots,t_{d-2},e,f,j} \right) \\
    = & \Bigg[ \frac{a+e \cdot (r+p)}{q} + \frac{e}{2q^2}+\frac{c}{lq} + \frac {t_1}{l^2q} + \dots + \frac{t_{d-2}}{l^{d-1}q}+\frac{b}{2l^{d-1}q^2}+\frac{j}{2l^dq^2}, \\
    & \quad \frac{a+e \cdot (r+p)}{q} + \frac{e}{2q^2}+\frac{c}{lq} + \frac {t_1}{l^2q} + \dots + \frac{t_{d-2}}{l^{d-1}q} +\frac{b}{2l^{d-1}q^2}+\frac{j+1}{2l^dq^2} \Bigg] \times \left[ 0,1 \right] 
\end{split}\end{equation*}
\item For all integers $0 \leq a < q$, , $0 \leq c < \frac{l}{2q}$, $0 \leq e < 2q$, $0 \leq f < l$ and $0 \leq t_i < l$ for $i=1,\dots, d-2$ we have in case of $0\leq b < q-1$ and $0 \leq e < q-1$
\begin{align*}
     \phi^{m_n \cdot \a_{n+1}} \circ \mathfrak{h}_{l_n,p_n,q_n,r_n} \left( A_{a,b,c,t_1, \dots,t_{d-2},e,f,j} \right) & = \phi^{m_n \cdot \a_{n+1}} \left( A_{a+e \cdot r_n, e, c,t_1, \dots,t_{d-2},b,j,f} \right) \\
     & = A_{a+(e+1) \cdot r_n, e+1, c,t_1, \dots,t_{d-2},b,j,f} \\
     & = \mathfrak{h}_{l_n,p_n,q_n,r_n}\left( A_{a,b,c,t_1, \dots,t_{d-2},e+1,f,j} \right)
\end{align*}
and in case of $q\leq b < 2q-1$ and $q \leq e < 2q-1$
\begin{align*}
     & \phi^{\left( m_n + 1 \right) \cdot \a_{n+1} - \frac{1}{q_{n+1}}} \circ \mathfrak{h}_{l_n,p_n,q_n,r_n} \left( A_{a,b,c,t_1, \dots,t_{d-2},e,f,j} \right) \\
     = & \phi^{\left( m_n + 1 \right) \cdot \a_{n+1} - \frac{1}{q_{n+1}}} \left( A_{a+e \cdot \left(r_n + p_n \right), e, c,t_1, \dots,t_{d-2},b,j,f} \right) \\
     = & A_{a+(e+1) \cdot \left(r_n + p_n \right), e+1, c,t_1, \dots,t_{d-2},b,j,f} \\
     = & \mathfrak{h}_{l_n,p_n,q_n,r_n}\left( A_{a,b,c,t_1, \dots,t_{d-2},e+1,f,j} \right).
\end{align*}
\end{enumerate}
\end{remark}

\subsection{Construction of the conjugation maps} \label{subsec:constr}
In \cite[section 5.1]{Ba-Ku} we proved the subsequent statement on the approximation of arbitrary permutations.

\begin{proposition}[\cite{Ba-Ku}, Theorem E] \label{permutation = block-slide}
Let $k,q,l \in \N$ and $\Pi$ be any permutation of $kql$ elements. We can naturally consider $\Pi$ to be a permutation of the partition $\mathcal{S}_{kq,l}$ of the torus $\T^2$. Assume that $\Pi$ commutes with $\phi^{1/q}$. Then $\Pi$ is a block-slide type of map.
\end{proposition}
We point out that the result is proven by going to a finer partition $\mathcal{S}_{kq,2l}$. This will play a role in the definition of ``good domains'' of our conjugation map in Definition \ref{dfn:goodarea}. \\

Since we know how to find real-analytic approximations of block-slide type of maps due to Proposition \ref{proposition approximation}, this result will be very useful in order to find a conjugation map providing the combinatorics from the previous subsection. In fact, we find a real-analytic diffeomorphism $h_{2, n}$ $\left(\varepsilon_n, \delta_n\right)$-close to $\mathfrak{h}_{l_n,p_n,q_n, r_n}$ involving the $x_1-x_2$-coordinates for our choices of $\delta_n =\frac{1}{n \cdot q_n}$ and $\varepsilon_n = \frac{\delta_n}{4d \cdot 2l^d_nq^2_n}$. 

\subsubsection*{The higher dimensional case}
In order to deal with higher dimensions, we consider the following step functions:
\begin{align}
&\psi_{i,l,q}^{(\mathfrak{1})}:[0,1)\to \R &\text{ defined by}\quad &\psi_{i,l,q}^{(\mathfrak{1})}(x)=\sum_{j=1}^{l-1}\frac{l-j}{l^{d+2-i}q}\chi_{[\frac{j}{l},\frac{j+1}{l}]}(x),\nonumber\\
&\psi_{i,l,q}^{(\mathfrak{2})}:[0,1)\to \R &\text{ defined by}\quad &\psi_{i,l,q}^{(\mathfrak{2})}(x)=\sum_{j=0}^{l^{d+2-i}q-1}\left(\frac{j}{l}-\Big\lfloor\frac{j}{l}\Big\rfloor \right)\chi_{[\frac{j}{l^{d+2-i}q},\frac{j+1}{l^{d+2-i}q}]}(x),\nonumber\\
&\psi_{i,l,q}^{(\mathfrak{3})}:[0,1)\to \R &\text{ defined by}\quad &\psi_{i,l,q}^{(\mathfrak{3})}(x)=\sum_{j=0}^{l-1}\frac{j}{l^{d+2-i}q}\chi_{[\frac{j}{l},\frac{j+1}{l}]}(x).\label{tilde psi}
\end{align}
Then we define the following three types of block slide map:
\begin{align*}
& \mathfrak{g}_{i,l,q}^{(\mathfrak{1})}:\T^d\to\T^d\qquad\text{defined by}\qquad\mathfrak{g}_{i,l,q}^{(\mathfrak{1})}\big(x_1,\ldots,x_d\big)= \left(x_1+\psi_{i,l,q}^{(\mathfrak{1})}(x_i),x_2,\ldots,x_d\right),\\
& \mathfrak{g}_{i,l,q}^{(\mathfrak{2})}:\T^d\to\T^d\qquad\text{defined by}\qquad \mathfrak{g}_{i,l,q}^{(\mathfrak{2})}\big(x_1,\ldots,x_d\big)=\left(x_1,\ldots,x_{i-1},x_i+\psi_{i,l,q}^{(\mathfrak{2})}(x_1),x_{i+1},\ldots,x_d\right),\\
& \mathfrak{g}_{i,l,q}^{(\mathfrak{3})}:\T^d\to\T^d\qquad\text{defined by}\qquad\mathfrak{g}_{i,l,q}^{(\mathfrak{3})}\big(x_1,\ldots,x_d\big)=\left(x_1-\psi_{i,l,q}^{(\mathfrak{3})}(x_i),x_2,\ldots,x_d\right).
\end{align*}
Note that the composition 
\begin{align}
\mathfrak{g}_{i,l,q}:\T^d\to\T^d\qquad\qquad\text{defined by}\qquad\mathfrak{g}_{i,l,q}=\mathfrak{g}_{i,l,q}^{(\mathfrak{3})}\circ\mathfrak{g}_{i,l,q}^{(\mathfrak{2})}\circ\mathfrak{g}_{i,l,q}^{(\mathfrak{1})}
\end{align}
maps the partition $\mathcal{G}_{i,l,q}$ to $\mathcal{G}_{i-1,l,q}$, where 
\begin{align}
& \mathcal{G}_{j,l,q}:=\Big\{\big[\frac{i_1}{l^{d+1-j}q},\frac{i_1+1}{l^{d+1-j}q}\big)\times \big[\frac{i_{2}}{l},\frac{i_{2}+1}{l}\big)\times\ldots\times\big[\frac{i_{j}}{l},\frac{i_{j}+1}{l}\big) \times \T^{d-j}:i_1 = 0,1,\ldots,l^{d+1-j}q-1,\nonumber\\
&\qquad\qquad\qquad\qquad\qquad\qquad\qquad\qquad\qquad\qquad (i_{2},\ldots,i_{j}) \in \{0,1,\ldots,l-1\}^{j-1}\Big\}.
\end{align}
So the composition
\begin{align}
\mathfrak{g}_{l,q}:\T^d\to\T^d\qquad\qquad\text{defined by}\qquad\mathfrak{g}_{l,q}=\mathfrak{g}_{2,l,q}\circ\ldots\circ\mathfrak{g}_{d,l,q}
\end{align}
maps the partition $\mathcal{G}_{l,q}$ to $\mathcal{T}_{l^dq}$.

In order to have conjugation maps that shift blocks of the same form, we modify the constructions slightly and define the block slide map $\mathfrak{g}_{i,l,q}^{(\mathfrak{2})}:\T^d\to\T^d$ for $i=3, \dots, d$ using the step function
\begin{equation}
\psi_{i,l,q}^{(\mathfrak{2})}:[0,1)\to \R \text{ defined by}\quad \psi_{i,l,q}^{(\mathfrak{2})}(x)=\sum_{j=0}^{2l^{d}q^2-1}\left(\Big\lfloor \frac{j}{2l^{i-2}q}\Big\rfloor-\Big\lfloor\frac{j}{2l^{i-1}q}\Big\rfloor \cdot l \right) \cdot \frac{1}{l} \cdot \chi_{[\frac{j}{2l^{d}q^2},\frac{j+1}{2l^{d}q^2}]}(x).
\end{equation}

Let $\varepsilon_n=\frac{\delta_n}{12d \cdot l^d_n \cdot q^2_n}$ and $\delta_n = \frac{1}{n \cdot q_n}$. With the aid of Proposition \ref{proposition approximation} we construct the real-analytic diffeomorphism $h_{i,l_n,q_n}^{(\mathfrak{j})}$ $\left(\varepsilon_n, \delta_n\right)$-approximating the maps $\mathfrak{g}_{i,l_n,q_n}^{(\mathfrak{j})}$ from above. Finally, we put
\begin{equation*}
h_{i,n} = h_{i,l_n,q_n}^{(\mathfrak{3})} \circ h_{i,l_n,q_n}^{(\mathfrak{2})}  \circ h_{i,l_n,q_n}^{(\mathfrak{1})} 
\end{equation*}
for $i = 3, \dots, d$.

\subsubsection*{Piecing everything together}
With the aid of the previous constructions we define our conjugation map
\begin{equation}
    h_{n} = h_{2,n} \circ h_{3,n}\circ \dots \circ h_{d,n} .
\end{equation}
We point out that $\phi^{\alpha_n} \circ h_{n} = h_{n} \circ \phi^{\alpha_n}$. \\

\begin{definition} \label{dfn:goodarea}
The ``good domain'' $G_n$ of $h_{n}$ and $h^{-1}_{n}$ respectively consists of all sets of the form
\begin{equation*}
    \left[\frac{i_1 + \delta_n}{2l^d_n q^2_n}, \frac{i_1 + 1- \delta_n}{2l^d_n q^2_n} \right] \times \left[\frac{i_2+\delta_n}{2l_n}, \frac{i_2+1-\delta_n}{2l_n} \right] \times \prod^d_{j=3}\left[\frac{i_j+\delta-n}{l_n}, \frac{i_j + 1 - \delta_n}{l_n} \right],
\end{equation*}
where $i_t \in \Z$ for $t=1, \dots, d$.
\end{definition}

\begin{remark}
On the ``good domain'' the deviation of the composition $h_{n}$ from the underlying block slide type of map is at most $\frac{\delta_n}{4 l^d_nq^2_n}$ by our choice of $\varepsilon_n$ in the construction of each map $h_{i,n}$.
\end{remark}

\section{Proof of the $(h,h+1)$-property}

\begin{figure}
    \centering
    \includegraphics[width=15cm, height=16cm, trim={1.5cm 8cm 6cm 5cm}, clip]{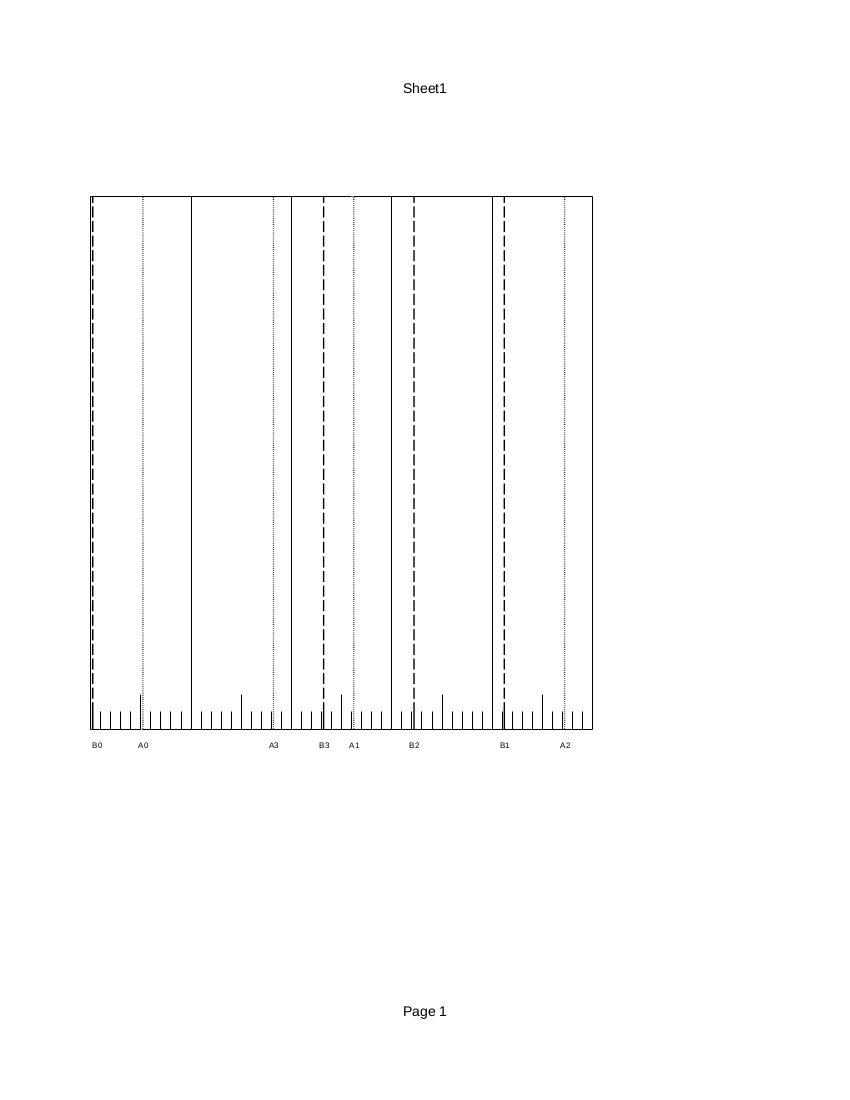}
    \caption{Sketch on $\T^2$ depicting the bases $\tilde{c}^{(n)}_{0,1}=A0\cup A1\cup A2\cup A3$ and $\tilde{c}^{(n)}_{0,2}=B0\cup B1\cup B2\cup B3$ of the two towers used in the approximation. This sketch is drawn with $p_n=3, q_n=5, m_n=3$. So $r_n=m_np_n=4 \mod q_n$. Our sketch is not fine enough to depict the $\d_n/(2l_n^2q_n^2)$ offset and the thickness of $q_n/q_{n+1}$. Note that the scale at the base is at a separation of $1/(2q_n^2)$. }
    \label{fig:t}
\end{figure}

\subsection{Construction of the towers} \label{subsection:tower1}
Let $\tilde{c}^{(n)}_{0,1}$ be the set
\begin{align*}
    \bigcup & \left[ \frac{i_1 \cdot r_n}{q_n} + \frac{i_1}{2q^2_n}+\frac{\delta_n}{2l^d_n q^2_n}, \frac{i_1 \cdot r_n}{q_n}+ \frac{i_1}{2q^2_n} + \frac{\delta_n}{2l^d_n q^2_n}+\frac{q_n}{q_{n+1}} \right] \\
    & \times \left[ \frac{i_2+\delta_n}{2l_n}, \frac{i_2+1-\delta_n}{2l_n} \right] \times \prod^d_{j=3} \left[ \frac{i_j+\delta_n}{l_n}, \frac{i_j+1-\delta_n}{l_n} \right], 
\end{align*}
where the union is taken over $i_1 \in \Z$, $0 \leq i_1 < q_n-1$, $i_2 \in \Z$, $0 \leq i_2 < 2l_n$, and $i_j \in \Z$, $0 \leq i_j < l_n$, for $j=3, \dots, d$ (see figure \ref{fig:t}). With this we define $c^{(n)}_{0,1} \coloneqq H^{-1}_n \left(\tilde{c}^{(n)}_{0,1} \right)$. By Remark \ref{remark combi} we have 
\begin{equation*}
    h^{-1}_{2,n} \left(\tilde{c}^{(n)}_{0,1} \right) \subset \left[0, \frac{1}{2l^{d-1}_nq_n}\right] \times \left[0, \frac{1}{l_n}\right] \times \prod^d_{j=3}\left[0,1\right].
\end{equation*}
Altogether, we have
\begin{equation*}
    h^{-1}_{n} \left(\tilde{c}^{(n)}_{0,1} \right) \subset \left[0, \frac{1}{l_nq_n}\right] \times \left[0, \frac{1}{l_n}\right] \times \dots \times \left[0, \frac{1}{l_n}\right].
\end{equation*}

Analogously we define the set $\tilde{c}^{(n)}_{0,2}$ by
\begin{align*}
    \bigcup & \left[ \frac{i_1 \cdot \left(r_n+p_n \right)}{q_n}+ \frac{1}{2q_n}+\frac{i_1}{2q^2_n} + \frac{\delta_n}{2l^d_n q^2_n}, \frac{i_1 \cdot \left(r_n + p_n \right)}{q_n} + \frac{1}{2q_n}+\frac{i_1}{2q^2_n}+ \frac{\delta_n}{2l^d_n q^2_n}+\frac{q_n}{q_{n+1}} \right] \\
    & \times \left[ \frac{i_2+\delta_n}{2l_n}, \frac{i_2+1-\delta_n}{2l_n} \right] \times \prod^d_{j=3} \left[ \frac{i_j+\delta_n}{l_n}, \frac{i_j+1-\delta_n}{l_n} \right] 
\end{align*}
where the union is taken over $i_1 \in \Z$, $0 \leq i_1 < q_n-1$, $i_2 \in \Z$, $0 \leq i_2 < 2l_n$, and $i_j \in \Z$, $0 \leq i_j < l_n$, for $j=3, \dots, d$. With this we define $c^{(n)}_{0,1} \coloneqq H^{-1}_n \left(\tilde{c}^{(n)}_{0,1} \right)$. \\

\begin{remark} \label{remark measure base}
We note 
\begin{equation*}
    \mu \left( c^{(n)}_{0,i}\right) = \mu \left( \tilde{c}^{(n)}_{0,i}\right) = \left(q_n - 1\right) \cdot l^{d-1}_n \cdot \frac{q_n}{q_{n+1}} \cdot \left(\frac{1-2\delta_n}{l_n}\right)^{d-1}  = \left(1-\frac{1}{q_n} \right) \cdot \left( 1- 2 \delta_n \right)^{d-1} \cdot \frac{q^2_n}{q_{n+1}}.
\end{equation*}
\end{remark}
With our number $m_n= \frac{q_{n+1}}{2q^2_n}$ we define the following sets:
\begin{align*}
&c^{(n)}_{i,1} \coloneqq T^{i}_n\left(c^{(n)}_{0,1}\right) \text{ for } i=0,1,...,m_n-1, \\
&c^{(n)}_{i,2} \coloneqq T^{i}_n\left(c^{(n)}_{0,2}\right) \text{ for } i=0,1,...,m_n.
\end{align*}
It should be noted that these sets are disjoint by construction. Hence, we are able to define these two towers:
\begin{definition}
The first tower $B^{(n)}_1$ with base $c^{(n)}_{0,1}$ and height $m_n$ consists of the sets $c^{(n)}_{i,1}$ for $i=0,1,...,m_n-1$.
The second tower $B^{(n)}_2$ with base $c^{(n)}_{0,2}$ and height $m_n+1$ consists of the sets $c^{(n)}_{i,2}$ for $i=0,1,...,m_n$.
\end{definition}
In the rest of this subsection we check that these towers satisfy the requirements of the definition of a $\left(h,h+1\right)$-approximation. First of all, we notice that both towers are substantial because we get by Remark \ref{remark measure base}:
\begin{align*}
& \mu\left(B^{(n)}_1\right)=m_n \cdot \mu\left(c^{(n)}_{0,1}\right) = \frac{q_{n+1}}{2q^2_n} \cdot \left(1-\frac{1}{q_n} \right) \cdot \left( 1- 2 \delta_n \right)^{d-1} \cdot \frac{q^2_n}{q_{n+1}} = \left(1-\frac{1}{q_n} \right) \cdot \left( 1- 2 \delta_n \right)^{d-1} \cdot \frac{1}{2}, \\
& \mu\left(B^{(n)}_2\right)=\left(m_n+1\right) \cdot \mu\left(c^{(n)}_{0,2}\right) = \left(  1+\frac{2 q^2_n}{q_{n+1}}\right) \cdot \left(1-\frac{1}{q_n} \right) \cdot \left( 1- 2 \delta_n \right)^{d-1} \cdot \frac{1}{2}.
\end{align*}
Using the notation from section \ref{section:theory} we consider the partial partition
\begin{equation*}
\xi_n \coloneqq \left\{ c^{(n)}_{i,1}, c^{(n)}_{k,2} \; : \; i=0,1,...,m_n-1; \;k=0,1,...,m_n\right\}
\end{equation*}
and $\sigma_n$ the associated permutation satisfying $\sigma_n \left( c^{(n)}_{i,1} \right) = c^{(n)}_{i+1,1}$ for $i=0,1,...,m_n-2$, $\sigma_n \left( c^{(n)}_{m_n-1,1} \right) = c^{(n)}_{0,1}$ as well as $\sigma_n \left( c^{(n)}_{k,2} \right) = c^{(n)}_{k+1,2}$ for $k=0,1,...,m_n-1$, $\sigma_n \left( c^{(n)}_{m_n,2} \right) = c^{(n)}_{0,2}$. We have to show: $\xi_n \rightarrow \varepsilon$ as $n\rightarrow \infty$. This property is fulfilled by the next lemma.

\begin{lemma}
We have
\begin{equation*}
\xi_n \rightarrow \varepsilon \text{ as } n\rightarrow \infty. 
\end{equation*}
\end{lemma}

\begin{proof}
The lemma is proven if we show that the partial partitions $\tilde{\xi}_n \coloneqq \left\{c \in \xi_n \;:\;\text{diam}\left(c\right) < \frac{1}{n}\right\}$ satisfy $\mu\left(\bigcup_{c \in \tilde{\xi}_n}c\right)\rightarrow 1$ as $n \rightarrow \infty$. For this purpose, we examine which tower elements satisfy the condition on their diameter. For this purpose, we note that 
\begin{equation*}
c^{(n)}_{i,j}= T^{i}_n\left(c^{(n)}_{0,j} \right)= H^{-1}_{n-1} \circ h^{-1}_n \circ \phi^{i \cdot \alpha_{n+1}} \circ H_n\left(H^{-1}_{n}\left(\tilde{c}^{(n)}_{0,j}\right)\right) = H^{-1}_{n-1} \circ h^{-1}_n \circ \phi^{i \cdot \alpha_{n+1}} \left(\tilde{c}^{(n)}_{0,j}\right)
\end{equation*}
for $j=1,2$ and the set $c^{(n)}_{i,j}$ is contained in one cuboid of $x_1$-length $\frac{1}{l_nq_n}$ and length $\frac{1}{l_n}$ in the $x_2, \dots, x_d$-coordinates, if  $\phi^{i \cdot \alpha_{n+1}} \left(\tilde{c}^{(n)}_{0,j}\right)$ belongs to the ``good domain'' of $h^{-1}_n$ (as introduced in Definition \ref{dfn:goodarea}) and the deviation $i \cdot \left| \a_{n+1}-\a_n \right|$ is less than $\frac{1}{2q^2_n}$. By requirement \ref{eq:l} on the number $l_n$ the diameter of the image of such a cuboid under $H^{-1}_{n-1}$ is less than $\frac{1}{n^2}$. So we have to check for how many iterates $i$ the set $\phi^{i \cdot \alpha_{n+1}} \left(\tilde{c}^{(n)}_{0,j}\right)$ is contained in the ``good domain'' of $h^{-1}_n$. Note that the bases of both towers are positioned in this ``good domain''. By the $\frac{1}{q_n}$-equivariance of the map $h_n$ and $i \cdot \alpha_{n+1} = \frac{i \cdot p_n}{q_n} + \frac{i}{q_{n+1}}$ we consider the displacement $\frac{i}{q_{n+1}}$, which is at most $\frac{1}{2q^2_n}$ because of $i \leq m_n \leq \frac{q_{n+1}}{2q^2_n}$. Since the projection on the $x_1$-coordinate of the ``good domain'' from Definition \ref{dfn:goodarea} consists of all intervals of the form $\left[\frac{i_1+\delta_n}{2l^d_nq^2_n}, \frac{i_1+1-\delta_n}{2l^d_nq^2_n} \right]$, we can estimate the number of allowed iterates $i \in \left\{0,1,...,m_n-1\right\}$ by $\left(1-2\delta_n\right)\cdot m_n$. We conclude that there are at least $2 \cdot \left(1-2\delta_n\right) \cdot m_n$ partition elements $c^{(n)}_{i,j}$ in $\tilde{\xi}_n$ and this corresponds to a measure
\begin{equation*}
\mu\left(\bigcup_{c \in \tilde{\xi}_n}c\right) \geq 2 \cdot \left(1-2\delta_n\right) \cdot m_n \cdot \mu\left(c^{(n)}_{i,j}\right) = \left(1-\frac{1}{q_{n}}\right) \cdot \left(1-2\delta_n\right)^d,
\end{equation*}
which converges to $1$ as $n\rightarrow \infty$.
\end{proof}

Actually we have a linked approximation of type $\left(h,h+1\right)$: Once again using the notation of section \ref{section:theory} we consider the partial partition
\begin{equation*}
\eta_n \coloneqq \left\{c^{(n)}_{i,1} \cup c^{(n)}_{i,2},\,c^{(n)}_{m_n,2} \;:\; 0\leq i \leq m_n-1\right\}
\end{equation*}
and prove $\eta_n \rightarrow \varepsilon$ as $n\rightarrow \infty$.

\begin{lemma} \label{lem:eta}
We have
\begin{equation*}
\eta_n \rightarrow \varepsilon \text{ as } n\rightarrow \infty. 
\end{equation*}
\end{lemma}

\begin{proof}
To see this we examine the partition 
\begin{equation*}
\tilde{\eta}_n \coloneqq \left\{\tilde{c}_i \coloneqq c^{(n)}_{i,1} \cup c^{(n)}_{i,2} \;:\; 0\leq i \leq m_n-1, \text{ diam}\left(\tilde{c}_i\right)<\frac{1}{n}\right\}
\end{equation*}
and show $\lim_{n \rightarrow \infty} \mu\left(\bigcup_{\tilde{c} \in \tilde{\eta}_n} \tilde{c}\right) = 1$. For this purpose, we note that $c^{(n)}_{i,1} \cup c^{(n)}_{i,2}$ is contained in one cuboid of $x_1$-length $\frac{1}{l_nq_n}$ and length $\frac{1}{l_n}$ in the $x_2, \dots, x_d$-coordinates, if  $\phi^{i \cdot \alpha_{n+1}} \left(\tilde{c}^{(n)}_{0,1} \cup \tilde{c}^{(n)}_{0,2}\right)$ belongs to the ``good domain'' of $h^{-1}_n$ and the deviation $i \cdot \left| \a_{n+1}-\a_n \right|$ is less than $\frac{1}{2q^2_n}$. Hence, the calculations from the previous lemma apply.
\end{proof}

\subsection{Speed of approximation}
In this section we want we prove that $T$ admits a good linked approximation of type $(h,h+1)$:
\begin{proposition} \label{prop:h+1}
The constructed diffeomorphism $T \in \text{Diff}^{\omega}_{\rho} \left(\T^d, \mu \right)$ admits a good linked approximation of type $(h.h+1)$.
\end{proposition}
Since $T$ admits a linked approximation by Lemma \ref{lem:eta}, we have to compute the speed of the approximation in order to prove this statement. First of all, we observe 
\begin{equation} \label{eq:firstofall}
\sum_{c \in \xi_n} \mu\left(T\left(c\right) \triangle \sigma_n\left(c\right)\right) \leq \sum_{c \in \xi_n}\left(\mu\left(T\left(c\right) \triangle T_n\left(c\right)\right)+ \mu\left(T_n\left(c\right) \triangle \sigma_n\left(c\right)\right)\right)
\end{equation}
recalling that $\sigma_n$ is the associated permutation satisfying $\sigma_n \left( c^{(n)}_{i,1} \right) = c^{(n)}_{i+1,1}$ for $i=0,1,...,m_n-2$, $\sigma_n \left( c^{(n)}_{m_n-1,1} \right) = c^{(n)}_{0,1}$ as well as $\sigma_n \left( c^{(n)}_{k,2} \right) = c^{(n)}_{k+1,2}$ for $k=0,1,...,m_n-1$, $\sigma_n \left( c^{(n)}_{m_n,2} \right) = c^{(n)}_{0,2}$. In the subsequent lemmas we examine each summand.
\begin{lemma} \label{lem:error1}
We have
\begin{equation*} 
\sum_{c \in \xi_n}\mu\left(T_n\left(c\right) \triangle \sigma_n\left(c\right)\right) \leq \frac{3 q_n}{q_{n+1}}.
\end{equation*}
\end{lemma}

\begin{proof}
We note $\sigma_n|_{c^{(n)}_{i,1}} = f_n|_{{c^{(n)}_{i,1}}}$ for $i=0,...,m_n-2$ and $\sigma_n\left(c^{(n)}_{m_n-1,1}\right) = c^{(n)}_{0,1}$ as well as $\sigma_n|_{c^{(n)}_{i,2}} = T_n|_{{c^{(n)}_{i,2}}}$ for $i=0,...,m_n-1$, $\sigma_n\left(c^{(n)}_{m_n,2}\right) = c^{(n)}_{0,2}$. \\ 
To estimate the expression $\sum_{c \in \xi_n}\mu\left(T_n\left(c\right) \triangle \sigma_n\left(c\right)\right)$ we recall equation \ref{a1} and Remark \ref{remark combi}, 3. Hereby, we observe that 
\begin{equation*}
    T_n\left(c^{(n)}_{m_n-1,1}\right) = H^{-1}_n \circ \phi^{\a_{n+1}} \circ H_n \circ H^{-1}_n \circ \phi^{(m_n -1) \cdot \a_{n+1}} \circ H_n \left( H^{-1}_n \left( \tilde{c}^{(n)}_{0,1} \right) \right) = H^{-1}_n \circ \phi^{m_n \cdot \a_{n+1}}\left( \tilde{c}^{(n)}_{0,1} \right)
\end{equation*} 
and $c^{(n)}_{0,1}=H^{-1}_n \left( \tilde{c}^{(n)}_{0,1}\right)$ differ in that part of $c^{(n)}_{0,1}$ corresponding to the value $i_1=q_{n}-2$ which is not mapped back to $c^{(n)}_{0,1}$ under $T^{m_n}_n$. This discrepancy yields a measure difference of $\frac{q_n}{q_{n+1}} \cdot \left(1-2\delta_n\right)^{d-1}$. \\
Analogously we recall equation \ref{a2} stating
\begin{equation*}
    \left(m_n + 1\right) \cdot \a_{n+1} = \frac{r_n+p_n}{q_n} + \frac{1}{2q^2_n} + \frac{1}{q_{n+1}}
\end{equation*}
and observe that $T_n\left(c^{(n)}_{m_n,2}\right)$ and $c^{(n)}_{0,2}$ differ in the part of $c^{(n)}_{0,2}$ corresponding to $i_1=q_{n}-2$ which is not mapped back to $c^{(n)}_{0,2}$. Additionally, for each other allowed value of $i_1$ there is the discrepancy of $\left(m_n +1 \right) \cdot \left( \a_{n+1}- \a_n \right)$ from $\frac{1}{2q^2_n}$ with value $\frac{1}{q_{n+1}}$. Thus, we get
\begin{equation*}
\mu\left(T^{m_n+1}_n\left(c^{(n)}_{0,2}\right) \triangle c^{(n)}_{0,2}\right) \leq \frac{2 q_n}{q_{n+1}}.
\end{equation*}
This yields
\begin{equation*} 
\sum_{c \in \xi_n}\mu\left(T_n\left(c\right) \triangle \sigma_n\left(c\right)\right) \leq \frac{3 q_n}{q_{n+1}}.
\end{equation*}
\end{proof}

In the next step we consider $\sum_{c \in \xi_n}\mu\left(T\left(c\right) \triangle T_n\left(c\right)\right)$: 
\begin{lemma} \label{lem:error2}
We have 
\begin{equation*} 
\sum_{c \in \xi_n}\mu\left(T\left(c\right) \triangle T_{n}\left(c\right)\right) \leq 40d \cdot \delta_{n+1}.
\end{equation*}
\end{lemma}

\begin{proof}
First of all we examine $\sum_{c \in \xi_n}\mu\left(T_{n+1}\left(c\right) \triangle T_n\left(c\right)\right)$. We compute for every $c\in \xi_n$ using the notation $\bar{c}\coloneqq H_{n}\left(c\right)=\phi^{i \cdot \a_{n+1}}\left(\tilde{c}^{(n)}_{0,j}\right)$:
\begin{align*}
\mu\left(T_{n+1}\left(c\right) \triangle T_{n}\left(c\right)\right) & = \mu\left(H^{-1}_{n+1} \circ \phi^{\alpha_{n+2}} \circ h_{n+1}\left(\bar{c}\right) \triangle H^{-1}_{n} \circ \phi^{\alpha_{n+1}}\left(\bar{c}\right)\right) \\
& = \mu\left(H^{-1}_{n+1} \circ \phi^{\alpha_{n+2}} \circ h_{n+1}\left(\bar{c}\right) \triangle H^{-1}_{n} \circ h^{-1}_{n+1} \circ \phi^{\alpha_{n+1}} \circ h_{n+1} \left(\bar{c}\right)\right) \\
& = \mu\left( \phi^{\alpha_{n+2}} \circ h_{n+1}\left(\bar{c}\right) \triangle \phi^{\alpha_{n+1}}\circ h_{n+1}\left(\bar{c}\right)\right).
\end{align*}

Since we have no control on $h_{n+1}\left(\bar{c}\right)$ for these areas of $\bar{c}$ that do not belong to the ``good domain'' of the map $h_{n+1}$, they will be part of the measure difference in our estimates. On the other hand, for the part of $\bar{c}$ belonging to the ``good domain'' of the map $h_{n+1}$ the difference is caused by the deviation $\left|\alpha_{n+2}-\alpha_{n+1}\right|$ and by the approximation error between the block slide type of map and its real-analytic approximation. By Definition \ref{dfn:goodarea} the ``good domain'' of the map $h_{n+1}$ on a block of the form
\begin{equation*}
    \Delta_{i_1, \dots, i_d} = \left[\frac{i_1}{2l^d_{n+1} q^2_{n+1}}, \frac{i_1 + 1}{2l^d_{n+1} q^2_{n+1}} \right] \times \left[\frac{i_2}{2l_{n+1}}, \frac{i_2+1}{2l_{n+1}} \right] \times \left[\frac{i_3}{l_{n+1}}, \frac{i_3 + 1 }{l_{n+1}} \right] \times \dots \times \left[\frac{i_d}{l_{n+1}}, \frac{i_d + 1}{l_{n+1}} \right]
\end{equation*}
is
\begin{equation*}
    \left[\frac{i_1 + \delta_{n+1}}{2l^d_{n+1} q^2_{n+1}}, \frac{i_1 + 1- \delta_{n+1}}{2l^d_{n+1} q^2_{n+1}} \right] \times \left[\frac{i_2+\delta_{n+1}}{2l_{n+1}}, \frac{i_2+1-\delta_{n+1}}{2l_{n+1}} \right] 
    \times \prod^d_{j=3}\left[\frac{i_j+\delta_{n+1}}{l_{n+1}}, \frac{i_j + 1 - \delta_{n+1}}{l_{n+1}} \right].
\end{equation*}
So its measure is 
\begin{equation*}
    \frac{\left(1-2\delta_{n+1}\right)^{d}}{4 \cdot l^{2d-1}_{n+1} \cdot q^2_{n+1}} \geq \frac{\left(1-2d \cdot \delta_{n+1}\right)}{4 \cdot l^{2d-1}_{n+1} \cdot q^2_{n+1}}.
\end{equation*}
By our choice of $\varepsilon_{n+1}$ the approximation error is at most $\frac{\delta_{n+1}}{4 l^{d}_{n+1}\cdot q^2_{n+1}}$ in each coordinate $x_2, \dots, x_d$. It follows that the measure difference of $\phi^{\alpha_{n+2}} \circ h_{n+1}\left(\bar{c}\right)$ and $\phi^{\alpha_{n+1}}\circ h_{n+1}\left(\bar{c}\right)$ on a set of the form $\Delta_{i_1, \dots,i_d}$ is at most
\begin{equation*}
2 \cdot \left(\frac{2 d \cdot \delta_{n+1}}{4 \cdot l^{2d-1}_{n+1} \cdot q^2_{n+1}} + \left|\alpha_{n+2}-\alpha_{n+1}\right| \cdot \frac{\left(1-2d \cdot \delta_{n+1}\right)}{4 \cdot l^{2d-1}_{n+1} \cdot q^2_{n+1}} + \frac{(d-1) \cdot \delta_{n+1}}{4 l^{3d-2}_{n+1} \cdot q^4_{n+1}} \right)
\leq \frac{3 d \cdot \delta_{n+1}}{l^{2d-1}_{n+1} \cdot q^2_{n+1}}.
\end{equation*}
Moreover, we recall that $\tilde{c}^{(n)}_{0,j}$ consists of at most $4 \cdot q^2_n \cdot l^{2d-1}_{n+1} \cdot q_{n+1}$ such blocks $\Delta_{i_1, \dots, i_d}$. Then we conclude
\begin{equation*}
\mu\left(T_{n+1}\left(c\right) \triangle T_n\left(c\right)\right) \leq 4 \cdot q^2_n \cdot l^{2d-1}_{n+1} \cdot q_{n+1} \cdot \frac{3 d \cdot \delta_{n+1}}{l^{2d-1}_{n+1} \cdot q^2_{n+1}} = q^2_n \cdot \frac{12 d \cdot \delta_{n+1}}{q_{n+1}}
\end{equation*} 
Each of the $\left(2m_n +1\right)$ elements $c \in \xi_n$ contributes and so we obtain
\begin{equation*}
\sum_{c \in \xi_n}\mu\left(f_{n+2}\left(c\right) \triangle f_{n+1}\left(c\right)\right) \leq 20d \cdot \delta_{n+1}.
\end{equation*}
Analogously estimating the other summands we observe
\begin{align*}
& \sum_{c \in \xi_n}\mu\left(T\left(c\right) \triangle T_{n}\left(c\right)\right) \\
\leq & \sum^{\infty}_{k=0}\sum_{c \in \xi_n} \mu\left(T_{n+k+1} \left(c \cap \bigcap^{k-1}_{j=0} H^{-1}_{n+j}\left(G_{n+j+1} \right)\right) \triangle T_{n+k} \left(c \cap \bigcap^{k-1}_{j=0} H^{-1}_{n+j}\left(G_{n+j+1} \right) \right) \right) \\
\leq & \sum^{\infty}_{k=n+1}20d \delta_{k}\leq 40d \cdot \delta_{n+1}.
\end{align*}
\end{proof}

\begin{proof}[Proof of Proposition \ref{prop:h+1}]
Using equation \ref{eq:firstofall} and the precedent three lemmas we conclude
\begin{equation*}
\sum_{c \in \xi_n} \mu\left(T\left(c\right) \triangle \sigma_n\left(c\right)\right) \leq \frac{3 \cdot q_n}{q_{n+1}} +  40d \cdot \delta_{n+1}.
\end{equation*}

In order to prove that this speed of approximation is of order $o\left(\frac{1}{m_n}\right)$ we compute
\begin{equation*}
\frac{\frac{3 \cdot q_n}{q_{n+1}} + 40d \cdot \delta_{n+1}}{\frac{1}{m_n}} = \frac{q_{n+1}}{2q^2_n} \cdot \left(\frac{3 \cdot q_n}{q_{n+1}} + \frac{40d}{(n+1) \cdot q_{n+1}} \right) 
= \frac{3}{2 \cdot q_n}+\frac{20d}{(n+1) \cdot q^2_n}
\end{equation*}
by our choice of $\delta_{n+1}=\frac{1}{(n+1) \cdot q_{n+1}}$. Since this converges to $0$ as $n\rightarrow \infty$, we have a good linked approximation of type $\left(h,h+1\right)$.
\end{proof}

\section{Proof of good cyclic approximation}

\subsection{Tower for good cyclic approximation}
Let $\tilde{d}^{(n)}_{0}$ be the set
\begin{equation*}
    \bigcup \left[ \frac{\delta_n}{2l^d_n q^2_n}, \frac{\delta_n}{2l^d_n q^2_n}+\frac{1}{q_{n+1}} \right] \times \left[ \frac{i_2+\delta_n}{2l_n}, \frac{i_2+1-\delta_n}{2l_n} \right] \times \prod^d_{j=3} \left[ \frac{i_j+\delta_n}{l_n}, \frac{i_j+1-\delta_n}{l_n} \right] ,
\end{equation*}
where the union is taken over $i_2 \in \N$, $0 \leq i_2 < 2l_n$, and $i_j \in \N$, $0 \leq i_j < l_n$, for $j=3, \dots, d$. With this we define $d^{(n)}_{0} \coloneqq H^{-1}_n \left(\tilde{d}^{(n)}_{0} \right)$ and the tower levels by
\begin{equation*}
d^{(n)}_i \coloneqq T^{i}_{n}\left(d^{(n)}_{0}\right) \ \ \text{ for } i=0,...,q_{n+1}-1.
\end{equation*}

Recall $\alpha_{n+1} = \frac{p_{n+1}}{q_{n+1}}$, where $p_{n+1}$ and $q_{n+1}$ are relatively prime. Hence, the tower levels are disjoint sets of equal measure not less than $\frac{\left(1-2\delta_n\right)^{d-1}}{q_{n+1}}$. Moreover, the associated cyclic permutation $\tilde{\sigma}_n$ is given by the description $\tilde{\sigma}_n |_{d^{(n)}_i} = T_n |_{d^{(n)}_i}$ for $i= 0,...,q_{n+1}-2$ and $\tilde{\sigma}_n\left(d^{(n)}_{q_{n+1}-1}\right)=d^{(n)}_0$. Since $T^{q_{n+1}}_n = \text{id}$ we also have $\tilde{\sigma}_n |_{d^{(n)}_{q_{n+1}-1}} = T_n |_{d^{(n)}_{q_{n+1}-1}}$. \\
In order to see that this provides a cyclic approximation of the constructed map $T$, we show that the partial partition $\Gamma_n \coloneqq \left\{d^{(n)}_i \; : \; i=0,...,q_{n+1}-1\right\}$ converges to the decomposition into points. 
\begin{lemma}
We have $\Gamma_n \rightarrow \varepsilon$ as $n \rightarrow \infty$.
\end{lemma}

\begin{proof}
It suffices to show that the partial partition $\tilde{\Gamma}_n \coloneqq \left\{ d^{(n)}_i \in\Gamma_n \;:\;\text{diam}\left(d^{(n)}_i\right)<\frac{1}{n}\right\}$ satisfies $\lim_{n\rightarrow \infty} \mu\left(\bigcup_{d \in \tilde{\Gamma}_n} d\right) =1$. As in the previous section we have to check for how many iterates $i \in \left\{0,1,...,q_{n+1}-1\right\}$ the set $\phi^{i \cdot \alpha_{n+1}}\left(\tilde{d}^{(n)}_{0}\right)$ is contained in the ``good domain'' of the map $h^{-1}_n$ whose corresponding length on the $x_1$-axis is at least $1-2\delta_n$ by Definition \ref{dfn:goodarea}. Since $\phi^{i \cdot \alpha_{n+1}}$ is equally distributed on $\mathbb{S}^1$, there are at least $\left(1-2\delta_n\right) \cdot q_{n+1}$ such iterates $i$. Then we conclude
\begin{equation*}
\mu\left(\bigcup_{d \in \tilde{\Gamma}_n} d\right) \geq \left(1-2\delta_n\right) \cdot q_{n+1} \cdot \mu\left(d^{(n)}_0\right) \geq \left(1-2\delta_n\right)^{d},
\end{equation*}
which converges to $1$ as $n\rightarrow \infty$.
\end{proof}

\subsection{Speed of approximation}
In this subsection we show that $T$ admits a good cyclic approximation.
\begin{proposition} \label{prop:cyclic}
The constructed diffeomorphism $T \in \text{Diff}^{\omega}_{\rho} \left(\T^d, \mu \right)$ admits a good cyclic approximation.
\end{proposition}

As observed in the previous subsection $\tilde{\sigma}_n |_d = T_n |_d$ for every $d \in \Gamma_n$. Thus, for the speed of approximation it holds:
\begin{equation} \label{eq:start}
\sum_{d \in \Gamma_n} \mu\left(T\left(d\right) \triangle \tilde{\sigma}_n\left(d\right)\right) = \sum_{d \in \Gamma_n}\mu\left(T\left(d\right) \triangle T_n\left(d\right)\right).
\end{equation}

\begin{lemma} \label{lem:errorB1}
We have
\begin{equation*} 
\sum_{d \in \Gamma_n}\mu\left(T\left(d\right) \triangle T_n\left(d\right)\right) \leq 20d \cdot \delta_{n+1}.
\end{equation*}
\end{lemma}

\begin{proof}
In order to estimate $\mu\left(T_{n+1}\left(d\right) \Delta T_{n}\left(d\right)\right)$ we denote $\bar{d}=H_n(d)$ and argue as in the proof of Lemma \ref{lem:error2} that we have no control on $h_{n+1}\left(\bar{d}\right)$ for these areas of $\bar{d}$ that do not belong to the ``good domain'' of the map $h_{n+1}$. Hence, these will be part of the measure difference in our estimates. On the other hand, for the part of $\bar{d}$ belonging to the ``good domain'' of the map $h_{n+1}$ the difference is caused by the deviations $\left|\alpha_{n+2}-\alpha_{n+1}\right|$ as well as the approximation error between the block slide type of map and its real-analytic approximation. \\
Using the estimate on the ``good domain'' of $h_{n+1}$ on sets of the form $\Delta_{i_1, \dots, i_d}$ in Lemma \ref{lem:error2} and that $H_{n}\left(d\right)$ consists of $4 \cdot l^{2d-1}_{n+1} \cdot q_{n+1}$ many of those sets $\Delta_{i_1, \dots, i_d}$, we obtain
\begin{equation*}
\mu\left(T_{n+1}\left(d\right) \triangle T_{n}\left(d\right)\right) 
\leq 4 \cdot l^{2d-1}_{n+1} \cdot q_{n+1} \cdot \frac{3 d \cdot \delta_{n+1}}{l^{2d-1}_{n+1} \cdot q^2_{n+1}} 
\leq \frac{12d \cdot \delta_{n+1}}{q_{n+1}}.
\end{equation*} 
Each of the $q_{n+1}$ elements $d \in \Gamma_n$ contributes and so we obtain
\begin{equation*}
\sum_{d \in \Gamma_n}\mu\left(T_{n+1}\left(d\right) \triangle T_{n}\left(d\right)\right) \leq 12d \cdot \delta_{n+1}.
\end{equation*}
Analogously we estimate the other summands and obtain
\begin{align*}
& \sum_{d \in \Gamma_n}\mu\left(T\left(d\right) \triangle T_{n}\left(d\right)\right) \\
\leq & \sum^{\infty}_{k=0}\sum_{d \in \Gamma_n} \mu\left( T_{n+k+1} \left( d \cap \bigcap^{k-1}_{j=0} H^{-1}_{n+j} \left( G_{n+j+1} \right) \right) \triangle T_{n+k} \left( d \cap \bigcap^{k-1}_{j=0} H^{-1}_{n+j} \left( G_{n+j+1} \right) \right) \right)\\
\leq & \sum^{\infty}_{k=n} 12d \delta_{k+1} \leq 20d \cdot \delta_{n+1}.
\end{align*}
\end{proof}

\begin{proof}[Proof of Proposition \ref{prop:cyclic}]
Using equation \ref{eq:start} as well as Lemma \ref{lem:errorB1} we conclude
\begin{equation*}
\sum_{d \in \Gamma_n} \mu\left(T\left(d\right) \triangle \tilde{\sigma}_n\left(d\right)\right) \leq 20d \cdot \delta_{n+1}.
\end{equation*}

In order to prove that this speed of approximation is of order $o\left(\frac{1}{q_{n+1}}\right)$ we compute
\begin{equation*}
\frac{20d \cdot \delta_{n+1}}{\frac{1}{q_{n+1}}} = q_{n+1} \cdot \frac{20d}{(n+1) \cdot q_{n+1}} = \frac{20d}{n+1}
\end{equation*}
by our choice of $\delta_{n+1}$. Since this converges to $0$ as $n\rightarrow \infty$, we have a good cyclic approximation.
\end{proof}

\end{document}